\documentclass[11pt]{amsart}

\usepackage{calc}  
\usepackage{amsmath}  
\usepackage{amsthm}
\usepackage{amsfonts}
\usepackage{mathrsfs}
\usepackage{amssymb}
\usepackage{amstext}
\usepackage{amscd}
\usepackage{enumerate}
\usepackage[all]{xy}
\newtheorem{thm}{Theorem}
\newtheorem{prop}[thm]{Proposition}
\newtheorem{cor}[thm]{Corollary}
\newtheorem{lem}[thm]{Lemma}

\theoremstyle{definition}
\newtheorem{fact}[thm]{Fact}

\newtheorem{dfn}[thm]{Definition}
\newtheorem{rem}[thm]{Remark}

\def\Z{\mathbb{Z}}
\def\Q{\mathbb{Q}}
\def\R{\mathbb{R}}
\def\C{\mathbb{C}}
\def\bbS{\mathbb{S}}

\def\bB{\mathbf{B}}
\def\bD{\mathbf{D}}

\def\bP{\mathbf{P}}

\def\bP{\mathbf{P}}

\def\bn{\mathbf{n}}

\def\cE{\mathcal{E}}
\def\cF{\mathcal{F}}

\def\cH{\mathcal{H}}

\def\cU{\mathcal{U}}

\def\cZ{\mathcal{Z}}
\def\bbS{\mathbb{S}}

\def\fC{\mathfrak{C}}

\def\fc{\mathfrak{c}}
\def\fd{\mathfrak{d}}
\def\ff{\mathfrak{f}}

\def\ft{\mathfrak{t}}
\def\fC{\mathfrak{C}}

\def\Cl{\operatorname{Cl}}

\def\dist{\operatorname{dist}}
\def\End{\operatorname{End}}

\def\Hom{\operatorname{Hom}}
\def\id{\operatorname{id}}
\def\im{\operatorname{im}}
\def\Image{\operatorname{Im}}
\def\Ind{\operatorname{Ind}}

\def\Inv{\operatorname{Inv}}

\def\Inv{\operatorname{Inv}}

\def\Ob{\operatorname{Ob}}
\def\Pic{\operatorname{Pic}}
\def\pr{\operatorname{pr}}

\def\SWF{\operatorname{SWF}}
\def\oI{\overline{I}}
\def\oL{\overline{L}}
\def\os{\overline{s}}

\begin{document}

\title[Gluing formula]
{Gluing formula for the stable cohomotopy version of Seiberg-Witten invariants along 3-manifolds with $b_1 > 0$}
\author{Hirofumi Sasahira}


\renewcommand{\thefootnote}{\fnsymbol{footnote}}
\footnote[0]{Partly supported by Grant-in-Aid for Young Scientists (B) 25800040}


\address{
Graduate school of Mathematics, Nagoya University,\endgraf 
Furocho, Chikusaku, Nagoya, Japan.
}
\email{hsasahira@math.nagoya-u.ac.jp}

\maketitle

\begin{abstract}
We will define a version of  Seiberg-Witten-Floer stable homotopy types for a closed, oriented 3-manifold $Y$ with $b_1(Y) > 0$ and a spin-c structure $\fc$ on $Y$ with $c_1(\fc)$ torsion under an assumption on $Y$. Using the Seiberg-Witten-Floer stable homotopy type, we will construct a gluing formula for the stable cohomotopy version of  Seiberg-Witten invariants of a closed 4-manifold $X$ which has a decomposition $X = X_1 \cup_Y X_2$ along $Y$.

\end{abstract}

\section{Main statements}

In \cite{Mano b1=0}, Manolescu constructed an invariant $\SWF(Y, \fc)$ for a 3-manifold $Y$ with $b_1(Y) = 0$ and a spin-c $\fc$ on $Y$, which is defined as an object of a $U(1)$-equivariant stable homotopy category $\fC$ and is called  the Seiberg-Witten-Floer stable homotopy type.  It is conjectured that the $U(1)$-equivariant homology of $\SWF(Y,\fc)$ is isomorphic to the Seiberg-Witten-Floer homology constructed by Kronheimer-Mrowka \cite{KM Three manifold}. As an application of the Seiberg-Witten-Floer stable homotopy type, we can define a relative invariant for an oriented, compact 4-manifold with boundary $Y$ which is a generalization of the stable cohomotopy version of Seiberg-Witten invariants for a closed 4-manifold due to Bauer and Furuta \cite{BF}.   Manolescu \cite{Manolescu Gluing} also constructed a gluing formula for the stable cohomotopy version of Seiberg-Witten invariants along a 3-manifold $Y$ with $b_1(Y)=0$,  which calculates the invariant of a closed 4-manifold in terms of the relative invariants. More recently, a Pin(2)-equivariant version of $\SWF(Y, \fc)$ is used to disprove the triangulation conjecture \cite{Mano triangulation} and to prove 10/8-type inequalities for 4-manifolds with boundary \cite{FL, Lin, Mano intersection form} which are generalization of \cite{Furuta Monopole, FK}.

In the case where $b_1(Y) > 0$, the construction of $\SWF(Y,\fc)$ was discussed by Kronheimer and Manolescu in \cite{KM ver 1}.  However Furuta pointed out that there is an obstruction for $\SWF(Y, \fc)$ to be well defined. In this paper, we will construct a version of Seiberg-Witten-Floer stable homotopy types for $Y$ with $b_1(Y) > 0$ and a spin-c structure $\fc$ with $c_1(\fc)$ torsion, provided that $Y$ satisfies a condition. Although we basically follow \cite{KM ver 1}, we modify in some points. In particular, we will make use of a spectral section of a family of Dirac operators on $Y$, which was introduced by Melrose and Piazza in \cite{MP}. Using the Seiberg-Witten-Floer stable homotopy type, we will define a relative invariant for a 4-manifold with boundary, and construct a gluing formula for the stable cohomotopy version of Seiberg-Witten invariants along a 3-manifold $Y$ with $b_1(Y) > 0$. The precise statements are the following.

Let $Y$ be a closed, oriented 3-manifold, $g$ be a Riemannian metric on $Y$ and $\fc$ be a spin-c structure on $Y$ with $c_1(\fc)$ torsion. We have a family of Dirac operators $\bD_{\fc} = \{ D_{A_h} \}_{[h] \in \Pic(Y)}$ on $Y$ parametrized by $\Pic(Y) = H^1(Y;\R)/H^1(Y;\Z)$. (See Section \ref{subsec spectral section}.)
Define $q_Y$ by
\[
   q_Y: \Lambda^3 H^1(Y;\Z) \rightarrow \Z, \
    c_1 \wedge c_2 \wedge c_3 \mapsto \left< c_1 \cup c_2 \cup c_3, [Y] \right>.
\]
Suppose $q_Y = 0$. Then the index $\Ind \bD_{\fc} \in K^1(\Pic(Y))$ of $\bD_{\fc}$ is trivial (\cite[Proposition 6]{KM ver 3}).
By a result of Melrose and Piazza \cite[Proposition 1]{MP}, we can take a spectral section $\bP = \{ P_{h} \}_{[h] \in \Pic(Y)}$ of $\bD_{\fc}$.

\begin{thm}
Let $Y$ be a closed 3-manifold, $g$ be a Riemannian metric on $Y$ and $\fc$ be a spin-c structure on $Y$.
If $c_1(\fc)$ is torsion and $q_Y = 0$, then we can define a Seiberg-Witten-Floer stable homotopy type $\SWF(Y,\fc, H, g, \bP)$ as an object in a stable category $\fC$. (See Section \ref{section category} for the definition of $\fC$.) Here $H$ is a submodule of $H^1(Y;\Z)$ of rank $b_1(Y)$ and $\bP$ is a spectral section of $\bD_{\fc}$.
\end{thm}

In this paper we do not discuss how $\SWF(Y,\fc, H, g,\bP)$ depends on $g$ and $\bP$.

For a closed, oriented 4-manifold $X$ and a spin-c structure $\hat{\fc}$ on $X$, we have the invariant  $\Psi_{X,\hat{\fc}}$ which is an element of $\pi^{b^+(X)}_{U(1)}(\Pic(X);\Ind \bD_{\hat{\fc}})$ due to Bauer and Furuta \cite{BF}. Here $\pi^{b^+(X)}_{U(1)}(\Pic(X);\Ind \bD_{\hat{\fc}})$ is a $U(1)$-equivariant stable cohomotopy group of the Thom space of the index bundle of Dirac operators on $X$ parametrized by the Picard torus $\Pic(X)$.
Let $\psi_{X,\hat{\fc}}$ be the restriction of $\Psi_{X, \hat{\fc}}$ to the fiber of $\Ind \bD_{\hat{\fc}}$. 
We can generalize the invariant $\psi_{X, \hat{\fc}}$ to a 4-manifold with boundary.

\begin{thm}
Let $Y$ be a closed, oriented 3-manifold with $q_Y = 0$. Take a Riemannian metric $g$, a spin-c structure $\fc$ on $Y$ with $c_1(\fc)$ torsion, a submodule $H$ of $H^1(Y;\Z)$ of rank $b_1(Y)$ and a spectral section $\bP$ of $\bD_{\fc}$.
Let $X_1$ be a compact, oriented 4-manifold with $\partial X_1 = Y$, $\hat{g}_1$ be a Riemannian metric on $X_1$ with $\hat{g}_1|_{Y} = g$ and $\hat{\fc}_1$ be a spin-c structure on $X_1$ with $\hat{\fc}_1|_{Y} = \fc$. We can define a relative invariant $\psi_{X_1, \hat{\fc}_1, H, g, \bP}$ which is an element of a $U(1)$-equivariant stable homotopy group of $\SWF(Y,\fc, H, g, \bP)$.
\end{thm}

Using the relative invariants, we can construct a gluing formula for $\psi_{X, \hat{\fc}}$.

\begin{thm} \label{thm gluing}
Let $Y$ be a closed, connected, oriented 3-manifold with $q_Y = 0$. (Note that we suppose that $Y$ is connected as in \cite{Mano errata}.) Take a Riemannian metric $g$ and a spin-c structure $\fc$ on $Y$ with $c_1(\fc)$ torsion, a submodule $H$ of $H^1(Y;\Z)$ generated by $\{ m_1 h_1, \dots, m_{b} h_{b} \}$ and a spectral section $\bP$ of $\bD_{\fc}$. Here $b = b_1(Y)$, $\{ h_1, \dots, h_{b} \}$ is a set of generators of $H^1(Y;\Z)$ and $m_j$ is a positive integer.
Let $X$ be a closed, oriented 4-manifold which has a decomposition $X = X_1 \cup_Y X_2$ for some compact oriented 4-manifolds $X_1$ and $X_2$ with boundary $Y$ and $-Y$. Suppose that we have a spin-c structure $\hat{\fc}$ on $X$ with $\hat{\fc}|_{Y} = \fc$ and that $m_j$ is sufficiently large for all $j$. Then we have
\[
           \psi_{X, \hat{\fc}} = 
             \eta \circ \left( \psi_{X_1, \hat{\fc}_1, H, g, \bP} \wedge
                              \psi_{X_2, \hat{\fc}_2, H, g, \bP} \right)             
\]
in the category $\fC$. Here $\hat{\fc}_j = \hat{\fc}|_{X_j}$ and $\eta$ is a S-duality morphism
\[
    \eta:\SWF(Y,\fc, g, \bP) \wedge \SWF(-Y, \fc, g, \bP) \rightarrow S^0.
\]

\end{thm}

\noindent
{\bf Acknowledgements.}
The author is grateful to Mikio Furuta for his suggestions. The author would like to thank Tirasan Khandhawit for information about the double Coulomb condition, and Yukio Kametani and Nobuhiro Nakamura for useful conversations. The author also would like to thank Ciprian Manolescu for sharing \cite{Mano errata} with the author.

\section{Conley index, mapping cone and duality}

\subsection{Conley index}
Let $\gamma$ be a smooth flow on a finite dimensional manifold $Z$. That is, $\gamma$ is a smooth map
\[
   \begin{array}{rcl}
      \gamma: Z \times \R & \rightarrow & Z  \\
                 (z, T)   & \mapsto     & \gamma(z,T) = z \cdot T
   \end{array}
\]
such that
\[
     z \cdot 0 = z, \quad z \cdot (T+T') = (z \cdot T) \cdot T'.
\]
For each subset $B \subset Z$, the maximal invariant set $\Inv(B)$ in $B$ is given by
\[
    \Inv(B) = \{ \ z \in Z \ | \ z \cdot \R \subset B \ \}.
\]
If $B \cdot \R \subset B$, $B$ is called an invariant set.

Let $S$ be a compact invariant set in $Z$. If there is a compact neighborhood $N$ of $S$ in $Z$ with $S = \Inv (N)$, then we say that $S$ is an isolated invariant set, and $N$ is called an isolating neighborhood of $S$.

\begin{fact}[\cite{Conley Isolated, Sala}]
Let $S$ be an isolated invariant set and $U$ be a neighborhood of $S$ in $Z$. There is a pair $(N, L)$ with the following properties:

\begin{enumerate}
   \item
    $N$ and $L$ are compact subspaces of $Z$ with $L \subset N$.

    \item
    $N$ is an isolating neighborhood of $S$ with $N \subset U$.
    
    \item
    Take any point $z \in N$. If  $z \cdot T_0 \not \in N$ for some $T_0 > 0$, there is a  positive number $T$ with $0 < T < T_0$ such that $z \cdot T \in L$.
    
    \item
    $L$ is positively invariant. That is, for $z \in L$ and $T > 0$ suppose that $z \cdot [0, T] \subset N$. Then $z \cdot [0, T] \subset L$.

\end{enumerate}

The pair $(N, L)$ is called an index pair of $S$.

\end{fact}

The choice of index pair $(N, L)$ of $S$ is not unique, however, the homotopy type of the pointed space $(N/L, [L])$ is unique up to canonical homotopy equivalence. Let $(N', L')$ be another index pair of $S$. We can define a homotopy equivalence $(N/L, [L]) \rightarrow (N'/L', [L'])$ as follows. Take a large positive number $T_0$ such that for any $T > T_0$ we have
\[
  \begin{split}
  &  z \cdot [-T, T] \subset N \backslash L \Rightarrow z \in  N' \backslash L', \\
  &  z \cdot [-T, T] \subset N' \backslash L' \Rightarrow z \in N \backslash L.
  \end{split}
\]
For $T > T_0$ define
\begin{equation} \label{eq f_T}
   \begin{array}{rrcl}
      f_T: &  N/L & \rightarrow & N'/L' \\
           &   z  & \mapsto & 
        \left\{
          \begin{array}{ll}
             z \cdot 3T & \text{if $z \cdot [0, 2T] \subset N \backslash L$, 
                            $z \cdot [T, 3T] \subset N' \backslash L'$,} \\
             * & \text{otherwise.}
          \end{array}
        \right.
   \end{array}
\end{equation}
Then we can see that $f_T$ is well defined, continuous and a homotopy equivalence from $(N/L, [L])$ to $(N'/L', [L'])$. See \cite[Section 4]{Sala} for details.

\begin{dfn}
Let $S$ be an isolated invariant set in $Z$ and $(N, L)$ be an index pair of $S$. We define the Conley index $I(S)$ of $S$ to be the homotopy type of $(N/L, [L])$.
\end{dfn}

\subsection{Attractor-repeller sequence} \label{subsec attractor}

Let $S$ be an isolated invariant set. A compact subset $A$ of $S$ is called an attractor in $S$ if there is a compact neighborhood $U$ of $A$ in $S$ with $A = \omega(U)$, and $A$ is called an repeller if $A = \omega^* (U)$. Here
\[
   \begin{split}
     &\omega(U) = \Inv ( \Cl(U \cdot [0, \infty)))
                = \bigcap_{T > 0} \Cl ( U \cdot [T, \infty) ), \\
     &\omega^*(U) = \Inv ( \Cl(U \cdot (-\infty,0]))
                  = \bigcap_{T < 0} \Cl ( U \cdot (-\infty, T] ).
   \end{split}
\]
For any $B \subset Z$, $\Cl(B)$ stands for the closure of $B$ in $Z$.

Let $A$ be an attractor in $S$ and put $A^* = \{ z \in S | \omega(z) \cap A = \emptyset \}$. Then $A^*$ is a repeller, called the complementary repeller of $A$. The pair $(A, A^*)$ is called an attractor-repeller pair in $S$.
We will construct index pairs for $S, A$ and $A^*$, following \cite[Section 3.2]{Cornea Homotopical}.  Let $S_1$ be the maximal attractor in $Z \backslash S$.
Let $S_2$ be  the set that consists of points on $A$, $S_1$ and trajectories in $Z$ originating at $A$. We can see that $S_2$ is also an attractor in $Z$. Lastly let $S_3$ be the set that consists of points on $S_2$, $A^*$ and trajectories in $S$ originating at $A^*$. Then $S_3$ is an attractor in $Z$. Denote by $R_j$ the complementary repeller of $S_j$ in $Z$.
We can take a Lyapunov function $f_j$ associated with $(S_j, R_j)$. (See p. 33 in \cite{Conley Isolated}.) That is, $f_i$ is a continuous function $Z \rightarrow [0,1]$ such that
\[
   \begin{split}
      &  f_{j}^{-1}(0) = S_j,  \\
      &  f_{j}^{-1}(1) = R_j \ and \\
      &  \text{$f_j$ is strictly decreasing on orbits in $Z \backslash (S_j \cup R_j)$.}
   \end{split}
\]
Take a real number $a_j \in (0, 1)$ for $j = 1, 2, 3$. Since $R_2 \subset R_1$, we can assume that
\[
     \{ \ z \in Z \ | \ f_2(z) \geq a_2 \ \} \subset
     \{ \ z \in Z \ | \ f_1(z) \geq a_1 \ \}. 
\]
Put
\begin{equation} \label{eq index pairs}
 \begin{split}
  N_S &:= \{ \ z \in Z \ | \ f_1(z) \geq a_1, \ f_3(z) \leq a_3 \ \}, \\
  L_S &:= f_1^{-1}(a_1) \cap N_{S}, \\
  N_A &:= \{ \ z \in Z \ | \ f_1(z) \geq a_1, \ f_2(z) \leq a_2, \ f_3(z) \leq a_3 \}, \\
  L_A &:= f_1^{-1}(a_1) \cap N_A \ (= L_S), \\ 
  N_{A^*} &:= \{ z \in Z | \ f_2(z) \geq a_2, \ f_3(z) \leq a_3 \ \}, \\
  L_{A^*} &:=  f_2^{-1}(a_2) \cap N_{A^*}.
 \end{split}
\end{equation}
We can see that $(N_S, L_S), (N_A, L_A)$ and $(N_{A^*}, L_{A^*})$ are index pairs for $S, A$ and $A^*$ respectively. Since $N_A \subset N_S$ and $L_A = L_S$, we have the inclusion
\[
    I(A) = N_A / L_A \stackrel{i}{\longrightarrow} I(S) = N_S / L_S.
\]
Note that we have a natural identification
\[
     N_{A^*}/L_{A^*} = N_S / N_A.
\]
Therefore we have the projection
\[
    I(S) = N_S / L_S = N_S / L_A \stackrel{j}{\longrightarrow} I(A^*) = N_S / N_A.
\]
Next we define a map
\[
      k:I(A^*) \longrightarrow \Sigma I(A).
\]
Here $\Sigma I(A)$ is the suspension of $I(A)$. For a topological space $W$ with base point $w_0$, the suspension of $W$ is defined by the following:
\[
    \Sigma W = [0,1] \times W / 
      \{ 0 \} \times W \cup  \{ 1 \} \times W \cup [0, 1] \times \{ w_0 \}.
\]
Define a function $s' = s'_{A^*}:N_{A^*} \rightarrow [0, \infty]$ by
\[
  s'(z) = \sup \{ \ T \geq 0 \ | \ z \cdot [0, T] \subset N_{A^*} \ \}
\]
and put
\begin{equation} \label{eq def s}
   s(z) = s_{A^*}(z) = \min \{ s'(z),  1 \}.
\end{equation}
By Lemma 5.2 of \cite{Sala}, the function $s$ is continuous. 
Define
\[
     \begin{array}{rrcl}
       k: & I(A^*) & \rightarrow & \Sigma I(A) \\
          &  z   & \mapsto     & (1-s(z), z \cdot s(z)).
     \end{array}
\]
We can see that $k$ is a well-defined and continuous map.
Thus we have a sequence
\begin{equation}  \label{eq ext tri}
   I(A) \stackrel{i}{\longrightarrow} I(S) 
        \stackrel{j}{\longrightarrow} I(A^*)
        \stackrel{k}{\longrightarrow} \Sigma I(A)
        \stackrel{\Sigma i}{\longrightarrow} \Sigma I(S)
        \stackrel{\Sigma j}{\longrightarrow} \cdots
\end{equation}
It is well known that this sequence is exact (\cite{Conley Isolated, Sala}) .  To see the exactness of the sequence, we will construct a homotopy equivalence from $\Sigma I(S)$ to $C(k)$  explicitly. Here $C(k)$ is the mapping cone of $k$. In general, for a continuous map $f:V \rightarrow W$ between topological spaces $V$ and $W$ with base points $v_0$ and $w_0$, the mapping cone $C(f)$ is defined by the following:
\begin{gather*}
  C(f) = [0, 1] \times V \coprod W / \sim , \\
   (1, v) \sim w_0, \ [0, 1] \times \{ v_0 \} \sim w_0, \ (0, v) \sim f(v) \
   (v \in V).   
\end{gather*}
Define a function $a:(0, 1] \times [0, 1] \rightarrow [0, 1]$ by
\[
     a(s, t) = 
        \left\{
          \begin{array}{cl}         
            \frac{t}{s} + 1 - \frac{1}{s} & \text{if $1-s \leq t \leq 1$}, \\
                            0             & \text{otherwise}.
          \end{array}  
        \right.
\]
Then we define $\varphi:\Sigma I(S) \rightarrow C(k) = C(I(A^*)) \cup_k \Sigma I(A)$ by
\begin{equation} \label{eq varphi}
    \varphi(t, z) =
       \left\{
         \begin{array}{ll}
          (a(s(z), t), z) \in C(I(A^*)) & \text{if $1-s(z) \leq t \leq 1$,} \\
          (t, z \cdot s(z)) \in \Sigma I(A) & \text{if $0 \leq t \leq 1-s(z)$.}
            \end{array}
          \right.
\end{equation}
Here we think of $s=s_{A^*}$ as a function $N_{S} = N_{A^*} \cup N_A \rightarrow [0,1]$ by putting $s(z) = 0$ for $z \in N_A$. We can easily prove the following.

\begin{lem}
The map $\varphi$ is well defined and continuous.
\end{lem}

Next we prove that $\varphi$ is a homotopy equivalence.

\begin{lem} \label{lem homotopy eq}
The map $\varphi$ is a homotopy equivalence. Moreover the following diagram is  homotopy commutative:
\begin{equation} \label{eq hom comm}
\begin{CD}
I(A^*) @>{k}>> \Sigma I(A) @>{\Sigma i}>> \Sigma I(S) 
                              @>{\Sigma j}>> \Sigma I(A^*) \\
@V{\id}VV @V{\id}VV @V{\varphi}VV @V{\id}VV \\
I(A^*) @>{k}>> \Sigma I(A) @>{i'}>> C(k) @>{p'}>> \Sigma I(A^*).
\end{CD}
\end{equation}
Here $i'$ and $p'$ are the inclusion and projection respectively.
\end{lem}

\begin{proof}
Define $\psi:C(k) \rightarrow \Sigma I(S)$ by
\begin{equation} \label{eq psi}
    \psi(t, z) = 
      \left\{
        \begin{array}{ll}
          (1 - (1-t) s(z), z \cdot s(z)) & \text{if $(t, z) \in C(I(A^*))$,} \\
          (t, z)                         & \text{if $(t, z) \in \Sigma I(A)$}.
        \end{array}
      \right.
\end{equation}
This is a well-defined and continuous map.
It is easy to see that $\psi \circ \varphi \sim \id$ and $\varphi \circ \psi \sim \id$. We can also see that the above diagram is homotopy commutative.
\end{proof}
    
Since the second row in (\ref{eq hom comm}) is exact,  we obtain:

\begin{cor}[\cite{Conley Isolated, Sala}]
The sequence (\ref{eq ext tri}) is exact.
\end{cor}

\subsection{Duality of mapping cones}
Let $S, A, A^*$ be an isolated invariant set, an attractor in $S$ and the complementary repeller of $A$ in $S$ respectively.
Let $\bar{\gamma}:Z \times \R \rightarrow Z$ be the inverse flow of $\gamma$. Hence $\bar{\gamma}(z, T) = z \cdot (-T)$.  Then $A$ is a repeller, $A^*$ is an attractor and $(A^*, A)$ is an attractor-repeller pair in $S$ with respect to $\bar{\gamma}$.
As before we can define a continuous function $\bar{s}:N_A \rightarrow [0,1]$. We also have a continuous map $\bar{k}:\bar{I}(A) \rightarrow \Sigma \bar{I}(A^*)$ defined by $\bar{k}(z) = (1-\bar{s}(z), z \cdot (-\bar{s}(z)))$. Here $\bar{I}$ stands for the Conley index with respect to the inverse flow $\bar{\gamma}$. Write $-\bar{k}$ for the map $\bar{I}(A) \rightarrow \Sigma I(A^*)$ defined by $(-\bar{k})(z) = ( \bar{s}(z), z \cdot (- \bar{s}(z)))$.
From now on, we assume that $Z$ is an $n$-dimensional sphere $S^n = \R^{n} \cup \{ \infty \}$.
The pairs $(i, \bar{j})$, $(j, \bar{i})$ and $(k, -\bar{k})$ are Spanier-Whitehead dual \cite{Cornea Homotopical}. Hence by \cite[Theorem (6.10)]{Spanier Function} we have a duality map
\[
    \eta_C:C(k) \wedge C(-\bar{k}) \rightarrow \Sigma^2 S^n = S^{n+2}.
\]
Our aim is to give an explicit expression of this map. According to \cite{Spanier Function}, using our notation, $\eta_C$ is given as follows. Choose duality maps
\begin{gather*}
     \eta_A:I(A) \wedge \overline{I}(A) \longrightarrow S^n, \\
     \eta_{A^*}: I(A^*) \wedge \overline{I}(A^*) \longrightarrow S^n.
\end{gather*}
Since $k$ and $-\bar{k}$ are Spanier-Whitehead dual, the following diagram is homotopy commutative:
\[
\begin{CD}
I(A^*) \wedge \overline{I}(A)
@>{k \wedge \id}>> 
\Sigma I(A) \wedge \overline{I}(A)  \\
@V{\id \wedge (- \bar{k})}VV @VV{\Sigma \eta_A}V \\
I(A^*) \wedge \Sigma \overline{I}(A^*) @>>{\Sigma \eta_{A^*}}> 
S^{n+1} = \Sigma S^n
\end{CD}
\]
Fix a homotopy $H$ between $\Sigma \eta_A \circ (k \wedge \id)$ and $\Sigma \eta_{A^*} \circ (\id \wedge (- \bar{k}))$. That is,
\[
   \begin{split}
     & H:[0,1] \times (I(A^*) \wedge \oI(A))  \rightarrow  S^{n+1} = \Sigma S^n, \\
     & H(0, \cdot) = \Sigma \eta_A \circ (k \wedge \id), \\
     & H(1, \cdot) = \Sigma \eta_{A^*} \circ (\id \wedge (-\bar{k}) ), \\
     & H(u, *) = * \quad ( \forall u \in [0, 1]  ).
   \end{split} 
\]
Take $(t, z) \in C(I(A^*))$, $(s, w) \in \Sigma I(A)$, $(s', w') \in C(\bar{I}(A))$, $(t', z') \in \Sigma \bar{I}(A^*)$, where $t, s, s', t' \in [0, 1]$. Then the duality map $\eta_C$ is defined by the following formula:
\begin{equation} \label{eq def of eta C}
   \begin{split}
      \eta_C( (t, z) \wedge (t', z')) &= (t, t', \eta_{A^*}(z \wedge z')), \\
      \eta_C( (s, w) \wedge (s', w')) &= (s', s, \eta_{A}(w \wedge w')),   \\
      \eta_C( (t, z) \wedge (s', w'))   
       & = \left\{
          \begin{array}{ll}
          (s', H(\frac{t}{2s'},z \wedge w'))  &
            \text{if $t \leq s', s' \not= 0$,} \\
           (t, H(1 - \frac{s'}{2t}, z \wedge w')) &
              \text{if $s' \leq t, t \not= 0$,}
          \end{array}
           \right. \\
     \eta_C((w,s) \wedge (z',t')) &= *.
   \end{split}
\end{equation}
To get the explicit expression of $\eta_C$, we need to choose $\eta_{A}$, $\eta_{A^*}$ and $H$ concretely.  

We can write $\eta_A$ as follows. (See \cite[Section 3]{Dold Puppe} and \cite[Section 2.5]{Manolescu Gluing}. See also \cite{McCord}.)
Assume that $S$ does not include $\infty$. We may suppose that $N_S$ lies in $\R^n \subset Z = S^n$.
Fix small positive numbers $\epsilon$ and $\delta$ with $0 < \epsilon < \delta \ll 1$.  Put 
\begin{equation} \label{eq N'}
   \begin{split}
     & N_{A}' = N_{A} - \overline{L}_{A} \times [0, \delta), \\
     & N_{A}'' = N_{A} - L_{A} \times [0, \delta).
   \end{split}
\end{equation}
Here $\overline{L}_{A} \times [0, \delta)$ stands for a neighborhood $\{ z \in N_A | \dist (z, \overline{L}_A) < \delta \}$ of $\overline{L}_A$ in $N_A$ which is homeomorphic to $L_{A} \times [0, \delta)$. Similarly for $L_A \times [0, \delta)$. Take continuous maps
\begin{equation} \label{eq m}
   m_1:N_A \rightarrow N_A', \
   m_2:N_A \rightarrow N_A''
\end{equation}
such that
\begin{equation} \label{eq m delta}
  \begin{split}
   &   \| w - m_1(w) \| < 2\delta, \ m_1(L_A) \subset L_A, \
            \dist(m_1( \oL_{A}), \oL_{A}) > \delta, \\
   &   \| w'- m_2(w') \| < 2\delta, \ m_2(\oL_A) \subset \oL_A, \
            \dist( m_2(L_A), L_A ) > \delta.
  \end{split}
\end{equation}
Define $\eta_A:I(A) \wedge \overline{I}(A) \rightarrow S^n$ by
\[
   \eta_A(w \wedge w') =
     \left\{
       \begin{array}{ll}
         m_1(w) - m_2(w') & \text{if $\| m_1(w) - m_2(w') \| < \epsilon$,} \\
           *             & \text{otherwise}.
       \end{array}
     \right.
\]
Here we think of $S^n$ as $D^n(\epsilon)/S^{n-1}(\epsilon)$. This map is well defined and a duality map of $I(A)$ and $\overline{I}(A)$.

Similarly $\eta_{A^*}:I(A^*) \wedge \oI(A^*) \rightarrow S^n$ is defined by
\[
     \eta_{A^*}(z \wedge z') =
       \left\{
         \begin{array}{ll}
          n_1(z) - n_2(z')  & \text{if $\| n_1(z) - n_2(z') \| < \epsilon$,} \\
                *             & \text{otherwise.}
         \end{array}
       \right.
\]
Here $n_1:N_{A^*} \rightarrow N_{A^*}'$ and $n_2:N_{A^*} \rightarrow N_{A^*}''$ are maps satisfying the conditions similar to (\ref{eq m delta}).

Finally we write $H$ explicitly. Put $M = \Sigma \eta_A \circ (k \wedge \id)$, $N = \Sigma \eta_{A^*} \circ (\id \wedge (-\bar{k}) )$. Then we have
\begin{gather*}
   M(z \wedge w') = \\
     \left\{
       \begin{array}{ll}
         (1-s(z), m_1(z \cdot s(z)) - m_2(w')) &
            \text{if $\| m_1(z \cdot s(z)) - m_2(w') \| < \epsilon$,} \\
            *            & \text{otherwise},
       \end{array}
     \right.  \\
  N(z \wedge w') =   \\ 
     \left\{
      \begin{array}{ll}
        (\os(w'), n_1(z) - n_2(w' \cdot (-\os(w')))) &
            \text{if $\| n_1(z) - m_2(w' \cdot (-\os(w'))) \| < \epsilon$,} \\
            *            & \text{otherwise},
      \end{array}
     \right.
\end{gather*}
We have to construct a homotopy between $M$ and $N$. The homotopy $H$ consists of four homotopies $H^j$ $(j=1, 2, 3, 4)$.

Define
\[
    H^1: [0,1] \times ( I(A^*) \wedge \oI(A) )  \longrightarrow S^{n+1} = \Sigma S^n
\]
by
\begin{gather} 
   H^1(u, z \wedge w') =   \nonumber \\
    \left\{
      \begin{array}{ll}
        ( 1-s(z),  \hat{m}^1(u, z, w'))  &  
              \text{if $\| \hat{m}^1(u, z, w') \| < \epsilon$}, \\
          *    & \text{otherwise.}
      \end{array}
    \right.         \label{eq H^1}
\end{gather}
Here 
\[
    \hat{m}^1(u, z, w') = m_1(z \cdot s(z)) - m_2( w' \cdot (-u \bar{s}(w'))).
\]

\begin{lem} \label{lem H1 well def}
The map $H^1$ is well defined.
\end{lem}

\begin{proof}
We need to show that $H^1(u, z \wedge w') = *$ if $z \in L_{A^*}$ or $w' \in \oL_{A}$. Let $z \in L_{A}$. Then $s(z) = 0$. Hence $H^1(u, z \wedge w') = *$ since the first component of $H^1$ is $1$.
Suppose that $w' \in \oL_{A}$. Then $\os(w') = 0$. If $s(z) = 1$, $H^1(u, z \wedge w') = *$ since the first component of $H^1$ is $0$. Suppose $s(z) < 1$. Then $z \cdot s(z)$ lies in $L_{A^*} \subset \oL_A$. Hence
\[
      \| m_1(z \cdot s(z)) - m_2(w \cdot (-u \bar{s}(w')) \|
       = \| m_1(z \cdot s(z)) - m_2(w) \| > \delta  > \epsilon
\]
by (\ref{eq m delta}).
Therefore $H^1(u, z \wedge w') = *$.
\end{proof}

Put 
\[
    \hat{m}^1(z, w) = \hat{m}^1(1, z, w) = m_1(z \cdot s(z)) - m_2(w' \cdot (-\bar{s}(w'))).
\]
By Lemma \ref{lem H1 well def}, $M$ is homotopic to
\begin{gather} 
    H^1(1, \cdot):I(A) \wedge \bar{I}(A^*) \rightarrow \Sigma S^n, \label{eq H^1 u=1} \\
    H^1(1, z \wedge w') = 
      \left\{ 
         \begin{array}{ll}
        (1-s(z), \hat{m}^1(z, w))  & \text{if $\| \hat{m}^1 (z, w) \| < \epsilon$}, \\
               *               & \text{otherwise}.
         \end{array}
      \right.           \nonumber
\end{gather}

\noindent
To define the second homotopy $H^2:[0,1] \times ( I(A^*) \wedge \oI(A) ) \rightarrow \Sigma S^{n}$, choose extensions $\tilde{m}_j:N_{S} \rightarrow N_{S}$ and $\tilde{n}_j:N_{S} \rightarrow N_{S}$ of $m_j$ and $n_j$. We may suppose that
\begin{equation} \label{eq tilde m n}
\begin{array}{c}
      \| \tilde{m}_j(z) - z \| < 2\delta, \quad
      \| \tilde{n}_j(z) - z \| < 2\delta  \quad
         (z \in N_S), \\
      f_3(\tilde{m}_1(z)) < a_3 \quad (z \in \oL_S), \\
      f_1(\tilde{n}_2(z)) > a_1 \quad
         (z \in L_{S}), \\
      \tilde{m}_1(z) = \tilde{n}_1(z) \quad (z \in L_S), \\
      \tilde{m}_2(z) = \tilde{n}_2(z) \quad (z \in \oL_S).
\end{array}
\end{equation}
and that $\tilde{m}_j$ and $\tilde{n}_j$ are homotopic to the identity of $N_{S}$. Here $f_1, f_3$ and $a_1, a_3$ are the Lyapunov functions and the positive numbers that appeared in (\ref{eq index pairs}). In particular, we have a homotopy
\begin{equation} \label{eq homotopy h_j}
  \begin{split}
    & h_j:[0,1] \times N_{S} \rightarrow N_{S},  \\
    & h_j(0, \cdot) = \tilde{m}_j, \quad h_j(1, \cdot) = \tilde{n}_j.
  \end{split}
\end{equation}
We may suppose that 
\[
      \| h_j(u, z) \| < 2\delta \quad
      ( u \in [0, 1], z \in N_S)
\]
and that
\begin{align}\
  &  h_1(u, z) = \tilde{m}_1(z) = \tilde{n}_1(z)  \quad
      (u \in [0,1], z \in L_S),      \label{eq f1 h1}   \\
  &  h_2(u, z) = \tilde{m}_2(z) = \tilde{n}_2(z)  \quad
     (u \in [0, 1], z \in \oL_{S}),    \label{eq f3 h2}  \\
  &  f_1(h_2(u,z)) > a_1 \quad
      (u \in [0,1], z \in N_S),       \label{eq f1 h2}         \\
  &  f_2( h_1(u,z)) < a_2 \quad
      ( u \in [0, 1), z \in L_{A^*}),   \label{eq f2 h1} \\
  &  f_2( h_2(u, z)) > a_2 \quad       
     ( u \in (0, 1], z \in L_{A^*}),     \label{eq f2 h2}\\
  &  f_3(h_1(u,z)) < a_3 \quad
       (u \in [0,1], z \in N_S).      \label{eq f3 h1 NS}  
\end{align}
 Note that 
\[
    h_1(u, L_{A^*}) \cap h_2(u, L_{A^*}) = \emptyset \ (\forall u \in [0, 1])
\]
by (\ref{eq f2 h1}) and (\ref{eq f2 h2}). Hence if $\epsilon > 0$ is small enough,
\begin{equation} \label{eq dist A^*}
    \dist( h_1(u, L_{A^*}), h_2(u, L_{A^*})) > \epsilon \ (\forall u \in [0, 1])
\end{equation}
since $L_{A^*}$ and the interval $[0, 1]$ are compact. Similarly by (\ref{eq f3 h1 NS}) and (\ref{eq f3 h2}) we have
\begin{equation} \label{eq dist A^*S}
   \dist ( h_1(u, N_S), h_2(u, \oL_S)) > \epsilon \ (\forall u \in [0, 1])
\end{equation}
if $\epsilon > 0$ is small enough.
Define $H^2:[0,1] \times (I(A^*) \wedge \bar{I}(A)) \rightarrow \Sigma S^{n}$ by
\begin{gather}
   H^2(u, z \wedge w') = 
     \left\{
      \begin{array}{ll}
        (1-s(z), \hat{h}(u, z, w')) & \text{if $\| \hat{h}(u, z, w') \| < \epsilon$}, \\
          *        & \text{otherwise},
      \end{array}
     \right.        \label{eq H^2}      \\
             \hat{h}(u, z, w') = 
                h_1(u, z \cdot s(z)) - h_2(u, w' \cdot (-\bar{s}(w'))).
   \nonumber
\end{gather}

\begin{lem}
The map $H^2$ is well defined.
\end{lem}

\begin{proof}
We need to show that $H(u, z \wedge w') = *$ if $z \in L_{A^*}$ or $w' \in \oL_A$.

Let $z \in L_{A^*}, w' \in N_A$. Then $s(z) = 0$ and the first component of $H^2$ is $1$. Hence $H^2(u, z \wedge w') = *$. 

Let $z \in N_{A^*}$ and $w' \in \oL_{A}$. Assume that $s(z) = 1$, then the first component of $H^2$ is $0$. Hence $H^2 = *$. Assume that $s(z) < 1$. We have $z \cdot s(z) \in L_{A^*}$.
If $w' \in L_{A^*} \subset \oL_A$, by (\ref{eq dist A^*}) we have
\[
     \| \hat{h}(u, z, w') \| > \epsilon.
\]
Hence we have $H^2(z \wedge w,u) = *$. Suppose that $w' \in \oL_{A} \backslash L_{A^*} \subset \oL_S$. By (\ref{eq dist A^*S}) we have
\[
    \| \hat{h}(u, z, w') \| > \epsilon.
\]
Therefore $H^2(u, z \wedge w') = *$.
\end{proof}

The map $H^2$ is a homotopy from (\ref{eq H^1 u=1})
to
\begin{gather}
   H^2(1, \cdot):I(A^*) \wedge \oI(A)  \longrightarrow  S^{n+1}    \label{eq H^2 u=1}  \\
    H^2(1, z \wedge w') =
       \left\{
        \begin{array}{ll}
         (1-s(z), \hat{n}(z, w)) & \text{if $\| \hat{n}(z, w) \| < \epsilon$}, \\
           *     & \text{otherwise},
        \end{array}
       \right.                   \nonumber
\end{gather}
Here 
\[
  \begin{split}
    \hat{n}(z, w) &= \tilde{n}_1(z \cdot s(z)) - \tilde{n}_2(w' \cdot \bar{s}(w')).
  \end{split}
\]

Define the third homotopy $H^3:[0,1] \times ( I(A^*) \wedge \oI(A)) \rightarrow \Sigma S^{n}$ by
\begin{gather}
  H^3(u, z \wedge w') =  \nonumber \\
    \left\{
     \begin{array}{ll}
       ( (1-u)(1-s(z)) + u \bar{s}(w'), \hat{n}(z, w)) &
         \text{if $\| \hat{n}(z, w) \| < \epsilon$,}  \\
       * & \text{otherwise}
     \end{array}
    \right.            \label{eq H^3}
\end{gather}

\begin{lem}
The map $H^3$ is well defined.
\end{lem}

\begin{proof}
If $s(z) = 0$, $\bar{s}(w')<1$ or $s(z) < 1$, $\bar{s}(w') = 0$ then
\[
    \| \hat{n}(z, w) \| > \epsilon
\]
as proved in the proof of the previous lemma. Hence $H^3(z \wedge w', u) = *$.
Assume that $s(z) = 0$ and $\bar{s}(w') = 1$. In this case, the first component of $H^3$ is $1$. Hence $H^3 = *$. Assume that $s(z) = 1$ and $\bar{s}(w') = 0$. Then the first component of $H^3$ is $0$. Hence $H^3 = *$.
\end{proof}

The map $H^3$ is a homotopy from (\ref{eq H^2 u=1}) 
to
\begin{gather} \label{eq H^3 u=1}
  H^3(1, \cdot): I(A^*) \wedge \oI(A)  \longrightarrow  \Sigma S^n  \\
  H^3(1, z \wedge w') =
       \left\{
        \begin{array}{ll}
         (\bar{s}(w'),  \hat{n}(z,w) ) & \text{if $\| \hat{n}(z, w) \| < \epsilon$}, \\
           *     & \text{otherwise}.
        \end{array}
       \right.           \nonumber
\end{gather}
Lastly define $H^4:[0,1] \times (I(A^*) \wedge \oI(A)) \rightarrow \Sigma S^n$ by
\begin{gather} 
  H^4(u, z \wedge w') =        \nonumber \\
    \left\{
     \begin{array}{ll}
       (\bar{s}(w'), \hat{n}(u, z, w)) &
         \text{if $\| \hat{n}(u, z, w) \| < \epsilon$,}  \\
       * & \text{otherwise}
     \end{array}
    \right.                   \label{eq H^4}
\end{gather}
Here 
\[
   \hat{n}(u, z, w) = 
   \tilde{n}_1(z \cdot (1-u)s(z)) - \tilde{n}_2( w' \cdot \bar{s}(w')).
\]

\begin{lem}
The map $H^4$ is well defined.
\end{lem}

\begin{proof}
The proof is similar to that of Lemma \ref{lem H1 well def}.
\end{proof}

The map $H^4$ is a homotopy between (\ref{eq H^3 u=1}) and $N$.
Thus we have the homotopy $H$ between $M$ and $N$:
\begin{equation}   \label{eq H1234}
   H(u, z \wedge w') =
    \left\{
      \begin{array}{ll}
      H^1(4u, z \wedge w') & \text{if $0 \leq u \leq \frac{1}{4}$}, \\
      H^2(4u-1, z \wedge w') &
                 \text{if $\frac{1}{4} \leq u \leq \frac{1}{2}$}, \\
      H^3(4u-2, z \wedge w') &
                 \text{if $\frac{1}{2} \leq u \leq \frac{3}{4}$}, \\
      H^4(4u-3, z \wedge w') &
                \text{if $\frac{3}{4} \leq u \leq 1$}.          
      \end{array}
    \right.
\end{equation}

Substituting the definitions of $\eta_A, \eta_A*$ and $H$ into (\ref{eq def of eta C}), we get the explicit formula for $\eta_C$. We use the formula to prove the following:

\begin{lem} \label{lem CD duality}
The following diagram is homotopy commutative:
\[
   \begin{CD}
     \Sigma I(S) \wedge \Sigma \bar{I}(S) @> {- \Sigma^2 \eta_S} >>    \Sigma^2 S^n \\
     @V\varphi \wedge \bar{\varphi}VV                     @|       \\
       C(k) \wedge C( -\bar{k})    @>\eta_{C}>>          \Sigma^2 S^n
   \end{CD}
\]
Here $\varphi$ and $\bar{\varphi}$ are the homotopy equivalences defined in Section \ref{subsec attractor}.
\end{lem}

\begin{proof}
From the construction of $\eta_A, \eta_{A^*}, H$ and the definition of $\eta_C$, we have
\begin{gather} 
     \eta_C ( \varphi(t,z) \wedge \bar{\varphi}(t',z') ) = 
     \label{eq muC varphi}    \\
     \nonumber
    \left\{
       \begin{array}{ll}
   (a(s(z), t), 1-t', \eta_{A^*}(z \wedge z' \cdot (-\bar{s}(z')) ))   &
     \text{if}
       \left\{
         \begin{array}{l}
           1 - s(z) \leq t \leq 1, \\
           0 \leq t' \leq 1 - \bar{s}(z')
         \end{array}
       \right.
          \vspace{2mm}  \\         
    (a(\bar{s}'(z'), t'), t, \eta_A(z \cdot s(z) \wedge z' ) ) &
     \text{if} 
     \left\{
       \begin{array}{l}
           0 \leq t \leq 1-s(z), \\
           1 - \bar{s}(z') \leq t' \leq 1,
       \end{array}
    \right.  
         \vspace{2mm} \\    
      \left( 
       a( \bar{s}(z'), t'),
         H ( A(\zeta), z \wedge z')
      \right) 
     & \text{if}
      \left\{
       \begin{array}{l}
         1 - s(z) \leq t \leq 1, \\
         1 - \bar{s}(z') \leq t' \leq 1, \\
         a(s(z), t) \leq a( \bar{s}(z'), t'), \\
         a(\bar{s}(z'), t') \not= 0,
       \end{array}
      \right.   
         \vspace{2mm}    \\
       \left(a(s(z), t), 
       H \left( A'(\zeta), z \wedge z' \right)  
       \right)  
     & \text{if}
      \left\{
       \begin{array}{l}
         1 - s(z) \leq t \leq 1, \\
         1 - \bar{s}(z') \leq t' \leq 1, \\
         a(\bar{s}(z'), t') \leq a(s(z), t), \\
         a(s(z), t) \not= 0,
       \end{array}
      \right.  
        \vspace{2mm}    \\
    * & \text{if} 
        \left\{
         \begin{array}{l}
      0 \leq t \leq 1-s(x),  \\
      0 \leq t' \leq 1-\bar{s}(z').
         \end{array}
        \right.
       \end{array}
     \right.
\end{gather}
Here
\[
     \zeta = (t, t', z, z'), \quad
     A(\zeta) = \frac{a(s(z), t)}{2a( \bar{s}(z'), t')}, \quad
     A(\zeta') = 1 - \frac{a( \bar{s}(z'), t' )}{2 a(s(z),t)}.
\]

We can write as
\begin{gather} \label{eq mu_C}
    \mu_C( \varphi(t,z) \wedge  \bar{\varphi}(t',z')  ) \\
  = \left\{
      \begin{array}{ll}
        (s_1(\zeta), s_2(\zeta), l_1(\zeta) - l_2(\zeta)) & 
          \text{if} 
          \left\{
          \begin{array}{l}
          \| l_1(\zeta) - l_2(\zeta) \| < \epsilon, \\
          1 - s(z) \leq t \leq 1 \ \text{or} \ 1-\bar{s}(z') \leq t' \leq 1,
          \end{array}
          \right.
          \\
           * & \text{otherwise.}
      \end{array}                   \nonumber
    \right.
\end{gather}
Here  $s_j$ is a continuous function of $\zeta$ with values in $[0,1]$, and $l_j:N_S \rightarrow N_S$ is a continuous map.  By the construction, we can write
\[
    l_1(\zeta) = h_1(t_1(\zeta), z \cdot \tau_1(\zeta))
\]
using the homotopy (\ref{eq homotopy h_j}) and some continuous functions $t_1(\zeta)$ and $\tau_1(\zeta)$, and similarly we can write
\[
    l_2(\zeta) = h_2(t_2(\zeta), z' \cdot \tau_2(\zeta)).
\]
On the other hand
\begin{gather*}
   \eta_S( z \wedge z' ) = \\
     \left\{
      \begin{array}{ll}
        \tilde{m}_1 (z) - \tilde{m}_2(z') & 
         \text{ if $\| \tilde{m}_1(z) - \tilde{m}_2(z') \| < \epsilon,$ } \\
       * & \text{otherwise,}
      \end{array}
     \right.
\end{gather*}
where $\tilde{m}_j$ is the extension of $m_j$ satisfying (\ref{eq tilde m n}).
Define a homotopy $H'$ by
\begin{gather*}
       H':[0,1] \times ( \Sigma I(S) \wedge \Sigma \bar{I}(S) ) \rightarrow \Sigma^2 S^n \\
       H'(u, \zeta) =  \\
          \left\{
             \begin{array}{ll}
              (s_1(\zeta), s_2(\zeta),   H_1'(u, \zeta) - H_2'(u, \zeta))
                & \text{if} 
                   \left\{
                    \begin{array}{l}
                      \| H_1'(u, \zeta) - H_2'(u, \zeta) \| < \epsilon,  \\
                      1-s(z) \leq t \leq 1 \ \text{or} \
                      1-\bar{s}(z') \leq t' \leq 1,
                    \end{array}
                   \right.         \vspace{2mm}   \\
                * &  \text{otherwise.}
             \end{array}
          \right.
\end{gather*}
Here
\[ 
   \begin{split}
     & H_1'(u, \zeta) =
         h_1( (1-u)t_1(\zeta), z \cdot (1-u)\tau_1(\zeta)  ),  \\
     & H_2'(u, \zeta) =
         h_2( (1-u) t_2 (\zeta), z' \cdot (1-u) \tau_2(\zeta)).    
   \end{split}
\]
We can see that $H'$ is homotopy from $\eta_C \circ (\varphi \wedge \bar{\varphi})$ to
\begin{gather*}
      H'(1, \cdot) : \Sigma I(S) \wedge \Sigma I(S) \rightarrow \Sigma^2 S^n   \\
      H'(1,  \zeta) = \\
        \left\{
          \begin{array}{ll}
            (s_1(\zeta), s_2(\zeta), \tilde{m}_1(z) - \tilde{m}_2(z') &
             \text{if}
              \left\{
              \begin{array}{l}
                 \| \tilde{m}_1(z) - \tilde{m}_2(z') \| < \epsilon, \\
                 1-s(z) \leq t \leq 1 \ \text{or} \
                 1-\bar{s}(z') \leq t' \leq 1.
              \end{array}
              \right.            
          \end{array}
        \right.
\end{gather*}
From (\ref{eq muC varphi}) we can see that $s_j(\zeta)$ is a function of $t, t'$, $s(z)$ and $\bar{s}(z')$.
For $u \in [0, 1]$, let $s_j(u, \zeta)$ be the function obtained from $s_j(\zeta)$, replacing $s(z)$ and $\bar{s}(z')$ by $s(u, z) = (1-u) s(z) + u$ and $\bar{s}(u, z') = (1-u) \bar{s}(z')$ respectively. 
Define $H'':[0,1] \times (\Sigma I(S) \wedge \Sigma \bar{I}(S))  \rightarrow \Sigma^2 S^n$ by
\begin{gather*}
       H''(u, \zeta) =   \\
         \left\{
           \begin{array}{ll}
             (s_1(u, \zeta), s_2(u, \zeta), \tilde{m}_1(z) - \tilde{m}_2(z')  ) &
               \text{if} \left\{
                \begin{array}{l}
                  \| \tilde{m}_1(z) - \tilde{m}_2(z') \| < \epsilon, \\
                  1-s_1(u, \zeta) \leq t \leq 1 \
                    \text{or} \\
                  1-s_2(u, \zeta) \leq t' \leq 1,
                \end{array} 
                \right.   \\
             * & \text{otherwise.}
           \end{array}
        \right.
\end{gather*}
We can show that $H''$ is well defined and a homotopy from $H'(1, \cdot)$ to $-\Sigma^2 \eta_S$.
\end{proof}

\section{Seiberg-Witten-Floer stable homotopy type}


\subsection{Definition of stable homotopy category} \label{section category}
Following \cite{Mano b1=0} and \cite{Mar}, we introduce a category $\fC$ which we will need to define the Seiberg-Witten-Floer stable homotopy type.  An object of $\fC$ is a triple $(Z, m, n)$, where $Z$ is a pointed $U(1)$-topological space which is homotopy equivalent to a $U(1)$-CW complex  and $m \in \Z$, $n \in \Q$. For objects $(Z, m, n)$ and $(Z', m', n')$ in $\fC$, the set of morphisms from $(Z, m, n)$ to $(Z', m', n')$ is empty if $n-n' \not\in \Z$ and is defined by
\[
     \{ (Z, m, n), (Z', m', n') \}^{S^1} =
       \lim_{ \substack{ k \rightarrow \infty \\ l \rightarrow \infty} } 
       [ \Sigma^{\R^{k} \oplus \C^{l}}  Z, \ \Sigma^{\R^{k+m-m'} \oplus \C^{l+n-n'}} Z']_{0}^{S^1}.
\]
if $n - n' \in \Z$.
In $\fC$, the suspensions $( \Sigma^{\R} Z, m, n  )$ and $(\Sigma^{\C} Z, m, n)$ are canonically isomorphic to $(Z, m-1, n)$ and $(Z, m, n-1)$ respectively. For an object $\cZ = (Z, m, n) \in \Ob (\fC)$, we denote $(Z, m+m', n+n')$ by $(\cZ, m', n')$.
 Let $E$ be a direct sum of a real vector space $E_{\R}$ and a complex vector bundle $E_{\C}$. We define an $U(1)$-action on $E$ by the multiplications on $E_{\C}$. We define a desuspension $\Sigma^{-E} \cZ$ of $\cZ$ by $E$ to be
\[
        (\Sigma^{E} Z, m+2\dim_{\R} E_{\R}, n+2\dim_{\C} E_{\C}).
\]
We can see that $\Sigma^{E} \Sigma^{-E}\cZ$ and $\Sigma^{-E} \Sigma^E \cZ$ are canonically isomorphic to $\cZ$.

\subsection{Chern-Simons-Dirac functional}

Let $Y$ be an oriented, closed 3-manifold and choose a Riemannian metric $g$ and a spin-c structure $\fc$ of $Y$ with $c_1(\fc)$ torsion, where $c_1(\fc)$ is the first Chern class of the determinant line bundle of $\fc$. Write $\bbS$ for the spinor bundle on $Y$ associated with $\fc$. Fix a flat connection $A_0$ on $\det \fc$. 
The Chern-Simons-Dirac functional is defined by the following formula:
\begin{gather*}
    CSD:V =  
    L^2_{k+\frac{1}{2}} \left( \left(  \sqrt{-1}\ker d^* \right) \oplus \Gamma (\bbS) \right)
           \rightarrow \R,    \\
    CSD(a, \phi) = 
      - \frac{1}{2} \left( \int_Y a \wedge  da + 
         \int_Y \left< \phi, D_{A_0+a} \phi \right> d\mu_g \right).
\end{gather*}
Here $d^*:\Omega^1(Y) \rightarrow \Omega^0(Y)$ is the adjoint of $d:\Omega^0(Y) \rightarrow \Omega^1(Y)$,  $D_{A_0+a}$ is the twisted Dirac operator associated with $A_0 + a$. The critical points of $CSD$ are monopoles on $Y$ and the gradient flows of $V$ are monopoles on $Y \times \R$. We have an action of $U(1) \times H^1(Y;\Z)$ on $V$ defined as follows. Fix a point $y_0 \in Y$. For $h \in H^1(Y;\Z)$ we have a smooth map $g:Y \rightarrow U(1)$ such that $g^{-1} dg = h$ and $g(y_0) = 1$. Here we have considered $h$ to be a harmonic 1-form on $Y$. For $(z, h) \in U(1) \times H^1(Y;\Z)$, $(a, \phi) \in V$, we define
\[
      (z, h) \cdot (a, \phi) = (a-2h, z g \phi).
\] 
Using the fact that $c_1(\fc)$ torsion, we can see that $CSD$ is invariant under the action of $H^1(Y;\Z)$.
The gradient vector field $\nabla CSD$ of $CSD$ is given by
\[
      \nabla CSD:V \rightarrow V, \
      \nabla CSD(a, \phi) = (*da + q(\phi), D_{A_0+a} \phi). 
\]
Here $q(\phi)$ is a 1-form on $Y$ defined by $\rho^{-1}(\phi \otimes \phi^* - \frac{1}{2} |\phi|^2 \id)$ and $\rho$ is the Clifford multiplication.

\subsection{Spectral section} \label{subsec spectral section}

We need a tool called a spectral section to define the Seiberg-Witten-Floer stable homotopy type, which was introduced by Melrose and Piazza \cite{MP}.  Let $\cH^1_g(Y)$ be the space of harmonic 1-forms on $Y$.
For each harmonic 1-form $h \in \cH^1_g(Y)$, put
\[
       A_h = A_0 - 2\sqrt{-1} h.
\]
We have a family of Dirac operators $\bD_{\fc} = \{ D_{A_{h}} \}_{[h] \in \Pic(Y)}$ on $Y$ parametrized by $\Pic(Y) = H^1(Y;\R)/H^1(Y;\Z)$:
\[
   \begin{array}{rrcl}
        \bD_{\fc}: &
        \frac{\cH^1_g(Y) \times \Gamma(\bbS)}{H^1(Y;\Z)} &
        \rightarrow &
        \frac{\cH^1_g(Y) \times \Gamma(\bbS) }{H^1(Y;\Z)}  \\
        & [ h, \phi] &
        \mapsto &
        [ h, D_{A_h} \phi ]
   \end{array}
\]
Here we have used $\cH^1_g(Y) \cong H^1(Y;\R)$. 

\begin{dfn}
Let $\bP = \{ P_h \}_{[h] \in \Pic(Y)}$ be a family of self-adjoint projections on $L^2(\bbS)$ parametrized by $\Pic(Y)$. (For each $h \in H^1(Y;\R)$, we have an operator $P_h$ and $P_h$ is equivariant with respect to the $H^1(Y;\Z)$-action.) We call $\bP$ a spectral section of $\bD_{\fc}$ if there is a smooth function $R:\Pic(Y) \rightarrow \R$ such that if $D_{A_h} u = \lambda u$ for some $\lambda \in \R$ then 
  \[
       P_h u = 
         \left\{
           \begin{array}{ll}
             u & \text{if $\lambda > R(h)$}, \\
             0 & \text{if $\lambda < - R(h)$}.
           \end{array}
         \right.
  \]
\end{dfn}

The family $\bD_{\fc}$ of Dirac operators on $Y$ defines the index $\Ind \bD_{ \fc }$ as an element of $K^1(\Pic(Y))$. See \cite{APS III}. Suppose that
\[
      q_Y:\Lambda^3 H^1(Y;\Z) \rightarrow \Z, \
           c_1 \wedge c_2 \wedge c_3 \mapsto \left< c_1 \cup c_2 \cup c_3, [Y] \right>
\]
is trivial. This is equivalent to the condition that $\Ind \bD_{\fc} = 0 \in K^1(\Pic(Y))$. See \cite[Proposition 6]{KM ver 3}.  By Proposition 1 of \cite{MP}, the vanishing of $\Ind \bD_{\fc} \in K^1(\Pic(Y))$ implies the existence of a spectral section $\bP$ of $\bD_{\fc}$. Fix a spectral section $\bP$ of $\bD_{\fc}$. According to \cite{MP} we can construct a family of self-adjoint smoothing operators $\bB^{\bP} = \{ B^{\bP}_h \}_{[h] \in \Pic(Y)}$ parametrized by $\Pic(Y)$ with the following property:

\begin{enumerate}
  \item
  The image of $B^{\bP}_h$ is included in a subspace of $\Gamma(\bbS)$ spanned by a finite  number of eigenvectors of $D_{A_h}$.
  
  \item
  $D^{\bP}_h = D_{A_h} + B^{\bP}_h$ is invertible.
  
  \item
  The operator $P_h$ is the Atiyah-Patodi-Singer projection onto the positive eigenspace of $D^{\bP}_h$.

\end{enumerate}

\subsection{Transverse double system} 
\label{ss transverse}

From now on we assume that $b_1(Y) = 1$ and $c_1(\fc)$ is torsion.
Following \cite[Section 4]{KM ver 1}, we introduce a transverse double system. 
For $R>0$, put
\[
    Str(R) = \{ \ (a, \phi) \in V \ | \
     \exists h \in H^1(Y;\Z), \| h \cdot ( a, \phi) \|_{L^2_k} \leq R \ \}.
\]
Let $h_1 \in H^1(Y;\Z)$ be a generator. 
We have a natural decomposition $V= \sqrt{-1} (h_1 \R \oplus \im d^*) \oplus \Gamma (\bbS)$, where we consider $h_1$ as a harmonic 1-form on $Y$.  Let $p:V \rightarrow \sqrt{-1} h_1 \R \cong \R$ be the projection.

\begin{dfn} \label{dfn transverse}
A transverse double system is a pair $(f_1, f_2)$ of smooth functions $f_1, f_2:Str(R) \rightarrow \R$ having the following properties:

\begin{enumerate}
 \item
 There is a positive number $M>0$ such that $f_i(y) < 0$ if $p(y) < -M$ and $f_i(y) > 0$ if $p(y) > M$.
 
 \item
 If $f_i(y) \geq 0$ then $f_i(h_1 \cdot y) \geq 0$ for $i = 1, 2$.
 
 \item
 If $f_1(y) = 0$ then $\left< \nabla CSD(y), \nabla f_1(y) \right> > 0$, and if $f_2(y) = 0$ then $\left< \nabla CSD(y), \nabla f_2(y) \right> < 0$.

\end{enumerate}
\end{dfn}

\begin{lem}[\cite{KM ver 1}]
There exists a transverse double system.
\end{lem}

The third condition in Definition \ref{dfn transverse} means that the zero set of $f_i$ and the gradient flow of $CSD$ intersect transversely. Since $CSD$ is invariant under the action of $H^1(Y;\Z)$, the intersection of the set $h_1^n \{ y \in Str(R) | f_i(x) = 0 \}$ and the gradient flow is also transverse for each $n \in \Z$.

Fix a transverse double system $(f_1, f_2)$ and put 
\[
     A_n = h_1^n \{ \ y \in Str(R) \ | \ f_1(y) \leq 0 \ \}, \
     B_n = h_1^n \{ \ y \in Str(R) \ | \ f_2(y) \leq 0 \ \}.
\]
It follows from the second property in Definition \ref{dfn transverse} that
\[
    A_n \subset A_{n+1}, \ B_n \subset B_{n+1}
\]
for every integer $n$. Let
\[
      U_n = A_{n+1} \backslash A_n, \
      V_n = B_{n+1} \backslash B_n.
\]
From the first condition in Definition \ref{dfn transverse}, $p(U_n)$ and $p(V_n)$ are bounded in $\R$. This means that $U_n$ and $V_n$ are bounded. Hence $U_n$ intersect only finite many $V_n$'s.
Without loss of generality, we may suppose that they are $V_{n+1}, V_{n+2},\dots, V_{n+N}$.
Put
\[
     W_n^i = U_n \cap V_{n+i}
\]
for $i=1, \dots , N$.

\subsection{Conley index of $W^i_n$}

As in the previous subsection, suppose that $b_1(Y) = 1$ and $c_1(\fc)$ is torsion. Note that when $b_1(Y)=1$, $q_Y$ is always trivial.
Fix a spectral section $\bP$ of $\bD_{\fc}$.
We can decompose the gradient vector field $\nabla CSD$ of $CSD$ as $l^{\bP}_0 + c^{\bP}_0$, where $l^{\bP}_0 = *d \oplus D^{\bP}_0:V \rightarrow V$, $c^{\bP}_0 = \nabla CSD - l^{\bP}_0:V \rightarrow V$ is a compact map, and $D^{\bP}_0 = D_{A_0} + B^{\bP}_0$.
Choose real numbers $\lambda, \mu$ with $\lambda < \mu$, and let $V_{\lambda}^{\mu} = V_{\lambda}^{\mu}(A_0,g,\bP)$ be the subspace of $V$ spanned by eigenvectors of $l^{\bP}_0$ with eigenvalues in $(\lambda, \mu]$. We denote the $L^2$-projection $V \rightarrow V_{\lambda}^{\mu}$ by $p_{\lambda}^{\mu}$. Let $\gamma_{\lambda}^{\mu} = \gamma_{\lambda}^{\mu, A_0, g, \bP}$ be the flow on $V_{\lambda}^{\mu}$ induced by $\nabla_{\lambda}^{\mu} CSD := l^{\bP}_0 + p_{\lambda}^{\mu}c^{\bP}_0 :V_{\lambda}^{\mu} \rightarrow V_{\lambda}^{\mu}$.

The maximal invariant set $\Inv(W_n^i \cap V_{\lambda}^{\mu}; \gamma_{\lambda}^{\mu})$ of $\gamma_{\lambda}^{\mu}$ in $W^i_n \cap V_{\lambda}^{\mu}$ lies in the interior of $W_n^i \cap V_{\lambda}^{\mu}$ when $R$, $-\lambda$ and $\mu$ are large enough. (See \cite{KM ver 1}.) This means that we can define the Conley index $I_{\lambda}^{\mu}(W_{n}^i) = I_{\lambda}^{\mu}(W_n^i;A_0, g, \bP)$ of $\Inv ( W_n^i \cap V_{\lambda}^{\mu}; \gamma_{\lambda}^{\mu} )$. As in \cite{Mano b1=0} we can show the following:

\begin{lem} \label{lem h.e. lambda mu}
For large  $-\lambda, \mu > 0$, $\Sigma^{-V_{\lambda}^0} I_{\lambda}^{\mu}(W_n^i)$ is independent of the choice of $\lambda, \mu$ up to canonical homotopy equivalence.
\end{lem}

\begin{proof}
We may suppose that $\lambda' < \lambda$ and $\mu' > \mu$. For each $t \in [0, 1]$, put
\[
     p_t = (1-t)p_{\lambda'}^{\mu'} + tp_{\lambda}^{\mu}:
        V \rightarrow V_{\lambda'}^{\mu'}.
\]
Here we have used the fact that $V_{\lambda}^{\mu} \subset V_{\lambda'}^{\mu'}$. Consider the flow $\gamma_t$ on $V_{\lambda'}^{\mu'}$ defined by the vector field
\[
      l + p_{t} c p_t:V_{\lambda'}^{\mu'} \rightarrow V_{\lambda'}^{\mu'}.
\]
It is easy to see that if $R>0$, $-\lambda, -\lambda', \mu$ and $\mu'$ are large enough, $W_{n}^i \cap V_{\lambda'}^{\mu'}$ is an isolating neighborhood of $\Inv (W^i_n \cap V_{\lambda'}^{\mu'}; \gamma_t)$ for any $t \in [0, 1]$. Hence we have the canonical homotopy equivalence
\[
    I_{\lambda'}^{\mu'}(W^i_n; \gamma_{\lambda'}^{\mu'}) 
    = I_{\lambda'}^{\mu'}(W^i_n;\gamma_0) 
    \stackrel{\sim}{\rightarrow} I_{\lambda'}^{\mu'}(W^i_n; \gamma_1)
\]
defined as (\ref{eq f_T}). The flow $\gamma_1$ is equal to the flow defined by $l|_{V'} \times (l + p_{\lambda}^{\mu} c )$ on $V_{\lambda'}^{\mu'} = V' \times V_{\lambda}^{\mu}$. Here $V'$ is the orthogonal complement of $V_{\lambda}^{\mu}$ in $V_{\lambda'}^{\mu'}$. Therefore we have
\[
        I_{\lambda'}^{\mu'}(W^i_n; \gamma_1) = 
           \Sigma^{V_{\lambda'}^{\lambda}} I_{\lambda}^{\mu}(W_n^i; \gamma_{\lambda}^{\mu})
\]
Thus we obtain a canonical isomorphism
\[
     \Sigma^{ -V_{\lambda'}^0 } I_{\lambda'}^{\mu'}(W_n^i) \stackrel{\cong}{\rightarrow}
        \Sigma^{ - V_{\lambda}^0 } I_{\lambda}^{\mu} (W_n^i).
\]
\end{proof}

We put
\[
        J(W_n^i) = J(W_n^i;A_0, g, \bP) := \Sigma^{-V_{\lambda}^{0}} I_{\lambda}^{\mu}(W_n^i)
         \in \Ob(\fC).
\]

\begin{rem}
If $-\lambda, \mu \gg 0$, the Conley index $I_{\lambda}^{\mu}(W_n^i;A_0, g, \bP)$ is independent of the choice of $\bP$ up to canonical homotopy equivalence since the image of $B^{\bP}_0$ is included in a finite number of eigenvectors of $D_{A_0}$.   Hence $J(W^i_n)$ depends on $\bP$ only through $V_{\lambda}^0 = V_{\lambda}^0(A_0, g, \bP)$.
\end{rem}

\subsection{Isomorphism between $J(W^i_n)$ and $J(W^i_{n+1})$}
\label{subsec Iso f}

In this subsection, we will see that $J(W_n^i)$ and $J(W_{n+1}^i)$ are canonically isomorphic to each other and write the isomorphism explicitly.
We have the isomorphism induced by the gauge transformation:
\begin{equation} \label{eq hz}
   \begin{array}{ccc}
      J(W^i_{n}; A_0, g, \bP)  & \stackrel{\cong}{\rightarrow}  
                                  & J(W^i_{n+1}; A_0 - 2 \sqrt{-1} h_1, g, \bP) \\
           y        &  \mapsto & h_1 y
   \end{array}
\end{equation}
Here $h_1$ is the fixed generator of $H^1(Y;\Z)$.
For $s \in [-1, 0]$, put $A_s := A_0 + 2s \sqrt{-1}h_1$. Write $\tilde{D}_s = D_{A_s} + B^{\bP}_{sh_1}$.
Fix $s \in [-1, 0]$.  We can find $-\lambda, \mu \gg 0$ such that $\lambda$ and $\mu$ are not an eigenvalue of $\tilde{D}_{s}$. Take $s' \in [0, 1]$ with $s<s'$, $| s - s'| \ll 1$. Then $\lambda$ and $\mu$ are still not an eigenvalue of $\tilde{D}_{s'}$, and the dimension $\dim V_{\lambda}^{\mu}(A_{ s'' }, g,  \bP)$ is independent of $s'' \in [s, s']$.  The restriction of the $L^2$-projection $p_{\lambda, s'}^{\mu}:V \rightarrow V_{\lambda}^{\mu}(s') = V_{\lambda}^{\mu}(A_{s'}, g, \bP)$ to $V_{\lambda}^{\mu}(s) = V_{\lambda}^{\mu}(A_s, g, \bP)$ gives an isomorphism
\[
    \tilde{f}_{s s'}:V_{\lambda}^{\mu}(s)
      \stackrel{\cong}{\rightarrow}
        V_{\lambda}^{\mu}(s').
\]

\begin{lem} \label{lem f T_0}
We can take $T_0 > 0$ independent of $\lambda$ and $\mu$ such that  for $T > T_0$, large $-\lambda, \mu$ and $s' > s$ with $|s-s'|$ small, we can define a $U(1)$-equivariant homotopy equivalence $\hat{f}_{s,s';T}:I_{\lambda}^{\mu}(W_{n+1}^i;A_s) \rightarrow I_{\lambda}^{\mu}(W_{n+1}^i;A_{s'})$.
\end{lem}

\begin{proof}
Let $(N, L)$ and $(N', L')$ be index pairs for $\Inv (W_{n+1}^i \cap V_{\lambda}^{\mu}(s))$ and $\Inv (W_{n+1}^i  \cap V_{\lambda}^{\mu}(s') )$ such that $N \subset W_{n+1}^i \cap V_{\lambda}^{\mu}(s)$, $N' \subset W_{n+1}^i \cap V_{\lambda}^{\mu}(s')$. 
Identifying $V_{\lambda}^{\mu}(s)$ and $V_{\lambda}^{\mu}(s')$ with $\tilde{f}_{ss'}$, we want to define $\hat{f}_{s,s':T}$ by the formula (\ref{eq f_T}).
We need to show that we can find $T_0 > 0$ independent of $\lambda, \mu$ such that for $T > T_0$, large $-\lambda, \mu > 0$ and $s' > s$ with $|s-s'|$ small, we have
\begin{gather}
    \begin{split} \label{eq z T N}
      y \cdot [0, T] \subset N \backslash L \Rightarrow  \tilde{f}_{ss'}(y) \in N' \backslash L', \\
      y \cdot [0, T] \subset N' \backslash L' \Rightarrow  \tilde{f}^{-1}_{ss'}(y) \in N \backslash L.
    \end{split}
\end{gather}
Suppose that the first condition in (\ref{eq z T N}) does not hold. Then there exist sequences $T_{\alpha}, -\lambda_{\alpha}, \mu_{\alpha} \rightarrow \infty$,  $s_{\alpha}' \searrow s $  and sequences $(N_{\alpha}, L_{\alpha})$, $(N_{\alpha}', L_{\alpha}')$ of index pairs of $\Inv(W^i_{n+1} \cap V_{\lambda_{\alpha}}^{\mu_{\alpha}}(s))$, $\Inv( W^i_{n+1} \cap V_{\lambda_{\alpha}}^{\mu_{\alpha}}(s') )$ with $N_{\alpha}, N_{ \alpha}' \subset W^i_{n+1}$ such that
\[
     \exists y_{\alpha} \in N_{s, \alpha}, \
     y_{\alpha} \cdot [0,T_{\alpha}] \subset N_{\alpha} \backslash L_{\alpha}, \
     \tilde{f}_{s s'_{\alpha}}(y_{\alpha}) \not\in N_{\alpha}' \backslash L_{\alpha}'.
\]
Since $y_{\alpha} \cdot [0, T_{\alpha}] \subset W^i_n$, the energy of the trajectory
\[
     \hat{x}_{\alpha}:[0, T_{\alpha}] \rightarrow V, \
     \hat{x}_{\alpha}(T) = y_{\alpha} \cdot T
\]
is bounded by a constant independent of $\alpha$. This implies that there is a subsequence $\alpha'$ such that $x_{\alpha'}$ converges to a finite energy trajectory
\[
      \hat{x}:[0, \infty) \rightarrow V
\]
on each compact set in $[0, \infty)$ (See \cite[Section 5]{KM Three manifold}) , and the limit $\hat{x}(\infty)$ is a critical point of $CSD$ in $W^{i}_{n+1}$. On the other hand, the condition that $\tilde{f}_{s s'_{\alpha}}(y_{\alpha'}) \not\in N_{\alpha'}' \backslash L_{\alpha'}'$ implies that the limit $\hat{x}(\infty)$ should be in $W^{j}_{m}$ for some $j> i$ and $m$. This is a contradiction.
The proof for the second condition in (\ref{eq z T N}) is similar.

\end{proof}

Taking desuspension, we get an isomorphism
\begin{equation} \label{eq h.e. A_s}
     f_{s s'}:J(W^i_{n+1}; A_s, g, \bP) 
         \stackrel{\cong}{\rightarrow}
           J(W^i_{n+1}; A_{s'}, g, \bP)
\end{equation}
in $\fC$. Here we have used the fact that the $L^2$-projection gives an isomorphism
\[
       V_{\lambda}^0(s) \stackrel{\cong}{\rightarrow} V_{\lambda}^0(s')
\]
since there is no spectral flow for the family $\{ \tilde{D}_{s''} \}_{s'' \in [s,s']}$.

\begin{lem} \label{lem indep lambda mu}
The morphism $f_{s s'}$ is independent of the choices of $\lambda$ and $\mu$ up to canonical homotopy.
\end{lem}

\begin{proof}
Take $\lambda' < \lambda \ll 0 \ll \mu < \mu'$ and suppose that $\lambda, \mu, \lambda', \mu'$ are not an eigenvalue of $D_{A_{s''}}$ for all $s'' \in [s, s']$.  It follows from the construction  of $\hat{f}_{s s'}$ that the following diagram is commutative up to canonical homotopy:
\[
   \begin{CD}
        I_{\lambda'}^{\mu'}(W^i_{n+1};A_s, g, \bP) 
           @>{\hat{f}_{s s'}'}>> 
                    I_{\lambda'}^{\mu'}(W^i_{n+1};A_{s'}, g, \bP) \\
          @V  VV              @VV V                          \\
        \Sigma^{V_{\lambda'}^{\lambda}(A_s)} I_{\lambda}^{\mu}(W^i_{n+1};A_s, g, \bP)
           @>>  { (p_{\lambda', s'}^{\lambda})^+ \wedge \hat{f}_{ss'}} > 
             \Sigma^{ V_{\lambda'}^{\lambda}(A_{s'}) } 
                I_{\lambda}^{\mu}(W^i_{n+1};A_{s'}, g, \bP)
   \end{CD}
\]
Here the columns are the homotopy equivalences obtained in the proof of Lemma \ref{lem h.e. lambda mu}, and $p_{\lambda', s'}^{\mu'}$ is an isomorphism from $V_{\lambda'}^{\lambda}(A_s)$ to $V_{\lambda'}^{\lambda}(A_{s'})$ induced by the $L^2$-projection. 
\end{proof}

Suppose that we have $-\lambda, \mu \gg 0$ such that $\lambda, \mu$ are not an eigenvalue of $\tilde{D}_{s''}$ for $s'' \in [s, s']$. Fix $s'' \in [s,s']$. Then we have two isomorphisms
\[
   \begin{split}
     & f_{s s''}:J(W^{i}_{n+1};A_s, g, \bP) \rightarrow  J(W^i_{n+1}; A_{s''}, g, \bP), \\
     & f_{s'' s'}: J(W^i_{n+1}; A_{s'}, g, \bP) \rightarrow  J(W^i_{n+1}; A_{s'}, g, \bP). 
   \end{split}
\]
Composing $f_{s s''}$ and $f_{s'' s'}$, we get an isomorphism
\[
    \begin{split}
     f_{s'' s'} \circ f_{s s''}:  
        J(W^{i}_n; A_{s}, g, \bP)  \rightarrow  J(W^i_n; A_{s'}, g, \bP).  
    \end{split}
\]

\begin{lem} \label{lem f s s'' s'}
In the above situation,  $f_{ss'}$ is canonically homotopic to $f_{s'' s'} \circ f_{s s''}$.
\end{lem}

\begin{proof}
The statement follows from the fact that the following  diagram is commutative up to canonical homotopy:
\begin{gather*}
      \xymatrix{
       I_{\lambda}^{\mu}(W_{n+1}^i;A_s) \ar[rr]^{ \hat{f}_{ss'} } \ar[dr]_{\hat{f}_{ss''}} & 
                       & I_{\lambda}^{\mu}(W_{n+1}^i;A_{s'}) \\
              & I_{\lambda}^{\mu}(W_{n+1}^i;A_{s''}) \ar[ru]_{ \hat{f}_{s'' s'} }  &
      }   
\end{gather*}   
\end{proof}

Let $\Delta = \{ s_0 = -1 < s_1 < s_2 < \cdots < s_{\ell} = 0 \}$ be a partition of the interval $[-1, 0]$ with $|s_j - s_{j+1}| \ll 1$ so that we have $-\lambda_j, \mu_j \gg 0$ which are not an eigenvalue of $\tilde{D}_{s}$ for $s \in [s_j, s_{j+1}]$. 
Suppose that $\lambda_j \geq \lambda_{j+1}$. Then we have
\[
   \begin{split}
    &  \Sigma^{2V_{\lambda_{j+1}}^{\lambda_j}(s_{j})}   
       \Sigma^{ V_{\lambda_j}^0(s_j) } I_{\lambda_j}^{\mu_j}(W^i_{n+1}; A_{s_j})   
       \stackrel{2p_{\lambda_{j+1}, s_{j+1}}^{\lambda_j} \wedge 
         p_{\lambda_j, s_{j+1}}^{0} \wedge \hat{f}_{s_j s_{j+1}}}{\longrightarrow} \\
      & \Sigma^{2 V_{\lambda_{j+1}}^{\lambda_j}(s_{j+1}) }
       \Sigma^{ V_{\lambda_{j}}^{0}(s_{j+1}) } I_{\lambda_j}^{\mu_j}(W^i_{n+1};A_{s_{j+1}}) 
        \stackrel{\cong}{\longrightarrow}    \\
      &  \Sigma^{ V_{\lambda_{j+1}}^{0}(s_{j+1}) } 
         \Sigma^{ V_{\lambda_{j+1}}^{\lambda_j}(s_{j+1}) }
           I_{\lambda_j}^{\mu_j}(W_{n+1}^i;A_{s_{j+1}})         
           \stackrel{\cong}{\longrightarrow}   \\
      &   \Sigma^{ V_{\lambda_{j+1}}^{0} (s_{j+1}) } 
          I_{\lambda_{j+1}}^{\mu_{j+1}}(W_{n+1}^i;A_{s_{j+1}}).
   \end{split}
\]
Similarly, if $\lambda_j < \lambda_{j+1}$, then we have
\[
      \Sigma^{ V_{\lambda_j}^0(s_j) } I_{\lambda_j}^{\mu_j}(W^i_{n+1};A_{s_j}) 
         \stackrel{\cong}{\longrightarrow}
      \Sigma^{ 2V_{\lambda_{j}}^{\lambda_{j+1}} (s_{j})  } 
       \Sigma^{ V_{\lambda_{j+1}}^{0}(s_{j+1}) } 
        I_{\lambda_{j+1}}^{\mu_{j+1}}(W^i_{n+1};A_{s_{j+1}}).
\]
Therefore we have a homotopy equivalence
\begin{equation*} 
      \Sigma^{2 V_+ } \Sigma^{ V_{\lambda_0}^0(-1) } 
         I_{\lambda_{ 0 }}^{\mu_0}(W^i_{n+1};A_{-1})
    \stackrel{\cong}{\rightarrow}
      \Sigma^{2 V_- } \Sigma^{ V_{\lambda_{\ell}}^0 (0) }
      I^{\mu_{\ell}}_{\lambda_{\ell}}(W^i_{n+1};A_{0}), 
\end{equation*}
where
\begin{gather*}
      V_{+} = 
        \bigoplus_{ j  \in J_+ }  V_{\lambda_{j+1}}^{\lambda_j}(s_{j}), \
      V_{-} =
        \bigoplus_{j \in J_- } V_{ \lambda_j }^{\lambda_{j+1}}(s_{j}),    \\
      J_+ = \{ \ j \ | \ 0 \leq j \leq \ell, \lambda_{j} \geq \lambda_{j+1} \ \}, \
     J_- = \{ \ j \ | \ 0 \leq j \leq \ell, \lambda_{j} < \lambda_{j+1} \ \}.
\end{gather*}
We may suppose that 
\[
     \lambda_0 = \lambda_{\ell}, \quad \mu_0 = \mu_{\ell}.
\]
We write $\lambda, \mu$ for $\lambda_0, \mu_0$.
Then we can see that $\dim (V_+)_{\R} = \dim (V_{-})_{\R}, \dim (V_+)_{\C} = \dim (V_-)_{\C}$.
Fix trivializations $\ft_{\pm}$ of $V_{\pm}$:
\[
             \ft_{+}:V_+ \stackrel{\cong}{\rightarrow} \R^{d} \oplus \C^{d'}, \quad
             \ft_{-}:V_- \stackrel{\cong}{\rightarrow} \R^{d} \oplus \C^{d'}.
\]
We get a homotopy equivalence
\begin{equation} \label{eq map -1 0}
     \Sigma^{ 2(\R^{d} \oplus \C^{d'})  } \Sigma^{ V_{\lambda}^0(-1) } 
     I_{\lambda}^{\mu}(W^i_{n+1}A_{-1})
    \stackrel{\cong}{\rightarrow}
      \Sigma^{2(\R^{d} \oplus \C^{d'}) } \Sigma^{ V_{\lambda}^0 (0) }
      I^{\mu}_{\lambda}(W^i_{n+1};A_{0}).
\end{equation}
Taking a desuspension of this map  we get an isomorphism
\[ 
      J(W^i_{n+1};A_{-1}, g, \bP) \stackrel{\cong}{\rightarrow} J(W^i_{n+1};A_0, g, \bP)
\]
in $\fC$.
Composing this with $h_1:J(W^i_n;A_0, g, \bP) \rightarrow J(W^i_{n+1};A_{-1}, g, \bP)$, we obtain
\[
    \ff:J(W^i_n;A_0, g, \bP) \stackrel{\cong}{\rightarrow} J(W^i_{n+1};A_0, g, \bP).
\]


\subsection{Definition of $\SWF(Y, \fc, g, \bP)$: The case $b_1(Y) = 1$ }
Fix $T > T_0$,  $\Delta = \{ s_0 = -1 < s_1 < \cdots < s_{\ell} = 0 \}$, $-\lambda_j, \mu_j \gg 0$, and trivializations $\ft_{\pm}$ to get $\ff$.
As in Section \ref{subsec attractor}, we have a morphism defined by using the flow:
\[
   \begin{split}
      J(W^i_n) \rightarrow 
         \Sigma J(W^{i+1}_{n-1} \cup W^{i+1}_{n}) 
       = \Sigma (J(W^{i+1}_{n-1}) \vee J( W_n^{i+1})).
   \end{split}
\]
Composing this morphism with 
\[
      J(W^{i+1}_{n-1}) \vee J( W_n^{i+1}) 
         \stackrel{\ff \vee \id }{\rightarrow}
      J(W^{i+1}_n)
\]
we get a morphism
\[
     k = k_n^i:J(W^i_n) \rightarrow \Sigma J(W_n^{i+1}).
\]
Note that $k = k_1 + k_2$ in $\fC$, where
\[
  \begin{split}
   & k_1:J(W_n^i) \rightarrow 
         \Sigma J(W^{i+1}_{n-1}) \stackrel{\ff}{\rightarrow}
         \Sigma J(W^{i+1}_n),   \\
   & k_2:J(W_n^i) \rightarrow \Sigma J(W^{i+1}_n).
  \end{split}
\]             
                               
First we define $\SWF(Y, \fc, g, \bP)$ in the case where $N = 2$, where $N$ is the number of $V_i$'s which intersect with $U_n$ as in Section \ref{ss transverse}.  As we have explained, we have the morphism
\[
     k:J(W^1_n) \rightarrow 
       \Sigma ( J(W^2_{n-1}) \vee J(W^2_n) )  \rightarrow
       \Sigma J(W^2_n).
\]
We define 
\[
      \SWF(Y,\fc, g, \bP) = \Sigma^{-1} C(k) \in \Ob(\fC).
\]
More precisely we define $\SWF(Y, \fc, g, \bP)$ using a continuous map $\hat{k}$ which represents $k$ as follows. Fix $T > T_0$, $\Delta, \lambda_0,\dots, \lambda_{\ell}, \mu_0, \dots, \mu_{\ell}$ with $\lambda_0 = \lambda_{\ell}, \mu_0 = \mu_{\ell}$ and $\ft_{\pm}$, then we get a continuous map
\[
      \hat{k}:\Sigma^{2(\R^{d} \oplus \C^{d'})} \Sigma^{V_{\lambda}^0} 
              I_{\lambda}^{\mu}(W^1_n)
              \rightarrow
              \Sigma^{2(\R^{d} \oplus \C^{d'})} \Sigma^{V_{\lambda}^0} 
              \Sigma I_{\lambda}^{\mu}(W^2_n)
\]
which represents the morphism $k$.

\begin{dfn}
We define
\[
    \begin{split}
   \SWF(Y,\fc, g, \bP) 
      &= \SWF(Y,\fc, g, \bP;  
         A_0,n, \Delta, \{ \lambda_j, \mu_j \}_{j}, \ft_{\pm}, f_1, f_2)       \\
      &:= 
   ( C( \hat{k} ), 
    2d + 2\dim (V_{\lambda}^0)_{\R} + 1 ,  
     2d' + 2 \dim (V_{\lambda}^0)_{\C} ) 
      \\
      & \in \Ob( \fC ).
    \end{split}
\] 
\end{dfn}

Next we consider the case $N=3$.
As before we have the morphism
\[
     k^1:J(W^1_n) \rightarrow \Sigma J(W^2_n).
\]
We will define a morphism
\[
    K: \Sigma^{-1} C(k^1) \rightarrow \Sigma J(W^3_{n})
\]
as follows. 
Take $y \in J(W^1_n)$. We can write $k^1(y) = (1-s(y), y') \in \Sigma J(W^3_n)$ with some $y' \in J(W^3_n)$. 
We can also write $( \Sigma k^2 )(1-s(y), y') = (1-s(y), 1-s'(y'), y'')$ with some $y'' \in J(W^3_n)$, where
\[
     k^2:J(W_n^2) \rightarrow \Sigma ( J(W_{n-1}^3) \vee J(W_n^3) )
                  \rightarrow \Sigma J(W_n^3).
\]
We define a morphism by
\[
     \begin{array}{rcl}
        C(J(W^1_n)) &\longrightarrow& \Sigma^2 J(W^3_{n}) \\
           (t, y)    &\longmapsto& (1 - (1-t) s(y), 1-s'(y'), y'').
     \end{array}
\]
We can see that this is well defined. When $t = 0$, this morphism coincides with $\Sigma k^2 \circ k^1$.  Hence the above morphism and $k^2$ induce a morphism
\[
            C(k^1) \longrightarrow \Sigma^2 J(W^3_{n}).
\]
Taking desuspension, we obtain 
\[
           K:\Sigma^{-1} C(k^1) \rightarrow \Sigma J(W^3_{n}).
\]

\begin{dfn}
We define
\[
   \begin{split}
    \SWF(Y, \fc, g, \bP) 
       =& \SWF(Y, \fc, g, \bP; 
              A_0, n, \Delta, \{ \lambda_j, \mu_j \}_j, \ft_{\pm},  f_1, f_2) \\
      :=& \Sigma^{-1} C(K) \in \Ob(\fC).
   \end{split}
\]
\end{dfn}

More precisely, we use a continuous map which represents $K$ to define $\SWF(Y,\fc, g, \bP)$ as in the previous case.
For any $N \geq 4$, we can define $\SWF(Y,\fc, g, \bP)$ in a similar way.
For $H \subset H^1(Y;\Z)$ with $H \not= \{ 0 \}$, we  also define a variant $\SWF(Y, \fc, H, g, \bP)$ as follows:

\begin{dfn}
Let $H \subset H^1(Y;\Z)$ be a subspace with $H \not= \{ 0 \}$. We can take $m h_1$ as a generator of $H$ for some $m \in \Z_{ > 0}$. We denote by $\SWF(Y, \fc, H, g, \bP)$ the object of $\fC$ obtained by replacing $h_1$ with $m h_1$ in the construction of $\SWF(Y, \fc, g, \bP)$.
\end{dfn}

We will prove the following in Section \ref{subsec proof of prop indep}:

\begin{prop} \label{prop indep 1}
The object $\SWF( Y, \fc, H, g,\bP; A_0, n, \Delta, \{ \lambda_j, \mu_j \}_j, \ft_{\pm},  f_1, f_2)$ of $\fC$ is independent of the choices of $A_0$, $n$, $\Delta$, $\{ \lambda_j, \mu_j \}_j$, $\ft_{\pm}$ and $(f_1, f_2)$ up to canonical isomorphism in $\fC$.
\end{prop}

\subsection{Commutativity of $\ff_i$ and $\ff_j$} \label{subsec fi fj}

We have defined the Seiberg-Witten-Floer stable homotopy type for a 3-manifolds with $b_1(Y) = 1$. Next we will extend the definition to the case $b_1(Y) \geq 2$.
Fix a Riemannian metric $g$ and a spin-c structure $\fc$ on $Y$ with $c_1(\fc)$ torsion. Suppose that $q_Y = 0$. Then we can take a spectral section $\bP$ of the family $\bD_{\fc}$ of Dirac operators  on $Y$ parametrized $\Pic(Y)$ as before.

Let $\{ h_1, \dots, h_{b} \}$ be a set of generators of $H^1(Y;\Z)$, where $b = b_1(Y)$. Take a transverse double system $(f_1^{j}, f_2^{j})$ with respect to $h_j$ for each $j$. As in the previous case, we can define an object $J(W^{i_1, \dots, i_b}_{n_1, \dots, n_b}) = J(W^{i_1, \dots, i_b}_{n_1, \dots, n_b};A_0, g, \bP)$ of $\fC$. Here $A_0$ is a fixed flat connection on $\det \fc$. We can also define an isomorphism
\[
        \ff_j:J(W^{i_1, \dots, i_b}_{n_1, \dots, n_j, \dots,  n_b}; A_0, g, \bP) \rightarrow
              J(W^{i_1, \dots, i_b}_{n_1, \dots, n_j+1, \dots,  n_b}; A_0, g, \bP)
\]
as in Section \ref{subsec Iso f}.
Before we begin the construction of $\SWF(Y, \fc, g, \bP)$, we discuss commutativity of $\ff_i$ and $\ff_j$. 
To simplify notation, we suppose $b_1(Y) = 2$ and consider $\ff_1$ and $\ff_2$. The morphism $\ff_1$ is represented by a continuous map $\hat{f}_1 \circ h_1$, and similarly $\ff_2$ is represented by $\hat{f}_2 \circ h_2$.  Here $\hat{f}_j$ is a continuous map constructed as in Section \ref{subsec Iso f}.
We will construct a homotopy from  $\Sigma^{2\tilde{V}_-} (\hat{f}_2 \circ h_2 \circ \hat{f}_1 \circ h_1)$ to $\Sigma^{2\tilde{V}_-'} ( \hat{f}_1 \circ h_1 \circ \hat{f}_2 \circ h_2 )$, where $\tilde{V}_-, \tilde{V}_-'$ are suitable finite dimensional vector spaces which are sums of real and complex vector spaces. In particular $\ff_2 \circ \ff_1$ is equal to $\ff_1 \circ \ff_2$ in $\fC$.

Let $h_1, h_2$ be generators of $H^1(Y;\Z)$. For $s_1, s_2 \in [-1, 0]$, put
\[
     A_{s_1, s_2} = A_0 + 2\sqrt{-1} s_1 h_1 + 2\sqrt{-1} s_2 h_2.
\]
Take $-1 = s_1(0) < s_1(1) < \dots < s_1(\ell_1) = 0$, $-1 = s_2(0) < s_2(1) < \cdots < s_2(\ell_2) = 0$ with $| s_1(i) - s_1(i+1) |, | s_2(j) - s_2(j+1) | \ll 1$ such that there are $-\lambda(i,j), \mu(i,j) \gg 0$ which are not an eigenvalue of $\tilde{D}_{s_1, s_2}$ for $(s_1, s_2) \in [s_1(i), s_{1}(i+1)] \times [s_2(j), s_2(j+1)]$. Here $\tilde{D}_{s_1, s_2} = D_{A_{s_1, s_2}} + B^{\bP}_{s_1, s_2}$. We may suppose that
\begin{equation} \label{eq lambda mu ell}
\begin{split}
    &  \lambda(0, j) = \lambda(\ell_1, j),  \ \lambda(i, \ell_2) = \lambda(i, \ell_2),  \\
    &  \mu(0,j) = \mu(\ell_1, j), \  \mu(i, \ell_2) = \mu(i, \ell_2)
\end{split}
\end{equation}
for each $i, j$. We write $\lambda, \mu$ for $\lambda(0,0), \mu(0,0)$ respectively. 

By definition, $\ff_1$ is the composition of $h_1:J(W^{i_1,i_2}_{n_1,n_2};A_{0,0},\bP) \rightarrow J(W^{i_1,i_2}_{n_1+1, n_2};A_{-1, 0}, \bP)$ and $f_1:J(W^{i_1,i_2}_{n_1+1,n_2};A_{-1,0},\bP) \rightarrow J(W^{i_1,i_2}_{n_1+1,n_2};A_{0,0}, \bP)$, and $f_1$ is represented by a continuous map
\[ 
    \begin{split}
       \hat{f}_1:
              &  \Sigma^{ 2V_{1, +} } \Sigma^{ V_{\lambda}^0 (A_{-1,0},  g,  \bP)}
                 I_{\lambda}^{\mu}(W^{i_1,i_2}_{n_1+1,n_2};A_{-1, 0})
                 \rightarrow      \\
             &   \Sigma^{ 2V_{1,-} } \Sigma^{ V_{\lambda}^{0}(A_{0,0}, g, \bP) }
                 I_{\lambda}^{\mu}(W^{i_1,i_2}_{n_1+1, n_2};A_{0,0}).
    \end{split}
\]
Under the assumption (\ref{eq lambda mu ell}), we can see that $\dim (V_{1,+})_{\R} = \dim (V_{1,-})_{\R}$, $\dim (V_{1,+})_{\C} = \dim (V_{1,-})_{\C}$.
Similarly, $\ff_2$ is the composition of $h_2:J(W^{i_1,i_2}_{n_1,n_2};A_{0,0},\bP) \rightarrow J(W^{i_1,i_2}_{n_1,n_2+1};A_{0,-1},\bP)$ and $f_2:J(W^{i_1,i_2}_{n_1,n_2+1};A_{0,-1},\bP) \rightarrow J(W^{i_1,i_2}_{n_1,n_2+1};A_{0,0}, \bP)$, and $f_2$ is represented by a continuous map
\[
   \begin{split}
     \hat{f}_2:
             &  \Sigma^{ 2V_{2,+} } \Sigma^{ V_{\lambda}^0 (A_{0, -1}, g, \bP)}
                I_{\lambda}^{\mu}(W^{i_1,i_2}_{n_1, n_2+1};A_{0, -1})
                \rightarrow   \\
             &  \Sigma^{ 2V_{2,-} } \Sigma^{ V_{\lambda}^{0}(A_{0,0}, g, \bP) }
                I_{\lambda}^{\mu}(W^{i_1,i_2}_{n_1, n_2+1};A_{0,0}).
    \end{split}
\]
As before, we have $\dim (V_{2,+})_{\R} = \dim (V_{2,-})_{\R}$, $\dim (V_{2,+})_{\C} = \dim (V_{2,-})_{\C}$.
The morphism $\ff_2 \circ \ff_1$ is represented by the following continuous map
\[
     \hat{f}_2 \circ h_2 \circ \hat{f}_1 \circ h_1 =
     \hat{f}_2 \circ \hat{f}_1' \circ h_2 \circ h_1.
\]
Here
\[
   \begin{split}
    \hat{f}_1'= h_2 \circ \hat{f}_1 \circ h_2^{-1}
    \end{split} 
\]
Similarly, $\ff_1 \circ \ff_2$ is represented by a continuous map
\[
      \hat{f}_1 \circ h_1 \circ \hat{f}_2 \circ h_2 =
      \hat{f}_1 \circ \hat{f}_2' \circ h_1 \circ h_2.
\]
Here
\[
      \begin{split}
    \hat{f}_2'= h_1 \circ \hat{f}_2 \circ h_1^{-1}
       \end{split}
\]
We have
\begin{equation} \label{eq map f_2 f_1}
   \begin{split}
     & \hat{f}_2 \circ \hat{f}_1':  \\
            & \Sigma^{2V_{2, +} \oplus 2 (h_2 V_{1,+}) } 
              \Sigma^{ V_{\lambda}^0 (A_{-1,-1}, g, \bP)}
                 I_{\lambda}^{\mu}(W^{i_1,i_2}_{n_1+1,n_2+1};A_{-1, -1})
                 \rightarrow      \\
             &   \Sigma^{ 2V_{2,+} \oplus 2 (h_2 V_{1,-}) } 
                 \Sigma^{ V_{\lambda}^{0}(A_{0,-1}, g, \bP) }
                 I_{\lambda}^{\mu}(W^{i_1,i_2}_{n_1+1,n_2+1};A_{0,-1})
                 \rightarrow     \\
            &   \Sigma^{ 2V_{2,-} \oplus 2 (h_2 V_{1,-})} 
                \Sigma^{ V_{\lambda}^{0}(A_{0,0}, g, \bP) }
                I_{\lambda}^{\mu}(W^{i_1,i_2}_{n_1+1,n_2+1};A_{0,0})
    \end{split}
\end{equation}
and
\begin{equation} \label{eq map f_1 f_2}
     \begin{split}
     & \hat{f}_1  \circ \hat{f}_2':  \\
            & \Sigma^{2V_{1, +} \oplus 2 (h_1 V_{2,+}) } 
              \Sigma^{ V_{\lambda}^0 (A_{-1,-1}, g, \bP)}
                 I_{\lambda}^{\mu}(W^{i_1,i_2}_{n_1+1,n_2+1};A_{-1, -1})
                 \rightarrow      \\
             &   \Sigma^{ 2V_{1,+} \oplus 2 (h_1 V_{2,-}) } 
                 \Sigma^{ V_{\lambda}^{0}(A_{-1,0}, g, \bP) }
                 I_{\lambda}^{\mu}(W^{i_1,i_2}_{n_1+1, n_2+1};A_{-1,0})
                 \rightarrow     \\
            &   \Sigma^{ 2V_{1,-} \oplus 2 (h_1 V_{2,-})} 
                \Sigma^{ V_{\lambda}^{0}(A_{0,0}, g, \bP) }
                I_{\lambda}^{\mu}(W^{i_1,i_2}_{n_1+1, n_2+1};A_{0,0}).
    \end{split}
\end{equation}
For each $(i,j)$, put
\[
   \begin{split}
      V_{1, + }(i,j) 
      &=  \left\{
        \begin{array}{ll}
           V_{\lambda(i+1,j)}^{ \lambda(i,j) }( A_{s_1(i), s_2(j)} ) &
             \text{if $\lambda(i, j) > \lambda(i+1,j)$,  }  \\
          0 & \text{otherwise,}
        \end{array}
       \right.   \\
      V_{1,-}(i,j)
      &= \left\{
          \begin{array}{ll}
           V_{ \lambda(i, j)  }^{ \lambda(i+1,j)  }(A_{s_1(i), s_2(j)}) &
           \text{if $\lambda(i,j) < \lambda(i+1,j)$} \\
          0 & \text{otherwise,}
          \end{array}
         \right. \\
      V_{2,+}(i,j)
      &= \left\{
          \begin{array}{ll}
           V_{\lambda(i,j+1)}^{\lambda(i,j)}(A_{s_1(i),s_2(j)  }) &
           \text{if $\lambda(i,j) > \lambda(i,j+1)$,}   \\
           0 & \text{otherwise,}   
          \end{array}
         \right.    \\
      V_{2,-}(i,j)
      &= \left\{
          \begin{array}{ll}
           V_{\lambda(i,j)}^{\lambda(i,j+1)}(A_{s_1(i), s_2(j)}) &
               \text{if $\lambda(i,j) < \lambda(i,j+1)$,}  \\
           0 & \text{otherwise.}
          \end{array}
         \right.   \\
   \end{split}
\]
We introduce the following sets of $(i,j)$:
\[
   \begin{split}
     \tilde{J}  
    &=   \{ \ (i,j) \ | \  0 \leq  i < \ell_1, \ 0 < j \leq \ell_2, \   \},  \\
     \tilde{J}' 
     &  =\{ \ (i,j) \ | \ 0 < i \leq \ell_1, \ 0 \leq j < \ell_2, \ \},   \\
     \tilde{\tilde{J}}
    &  = \{ \ (i,j) \ | \  0 \leq  i \leq \ell_1, \ 0 \leq j \leq \ell_2, \ \}.  \\
   \end{split}
\]
Lastly, we define finite dimensional vector spaces:
\[
   \begin{split}
  \tilde{V}_{-} 
       & = \bigoplus_{ (i,j) \in \tilde{J} } V_{1,-}(i,j) \oplus V_{2,-}(i,j),  \\
  \tilde{V}_{-}' 
         & = \bigoplus_{ (i,j) \in \tilde{J}' } V_{1,-}(i,j) \oplus V_{2,-}(i,j),  \\
  \tilde{\tilde{V}}_- 
      & = \bigoplus_{ (i,j) \in \tilde{\tilde{J}} } V_{1,-}(i,j) \oplus V_{2,-}(i,j).
   \end{split}
\]
Taking the suspension of (\ref{eq map f_2 f_1}) by $2 \tilde{V}_-$, we get
\begin{equation} \label{eq tilde V f2 f1}
   \begin{split}
     & \Sigma^{2 \tilde{V}_- } \hat{f}_2  \circ \hat{f}_1': \\
     &  \Sigma^{2 \tilde{V}_- \oplus  2V_{2, +} \oplus 2 (h_2 V_{1,+}) } 
              \Sigma^{ V_{\lambda}^0 (A_{-1,-1}, g, \bP)}
                 I_{\lambda}^{\mu}(W^{i_1,i_2}_{n_1+1, n_2+1};A_{-1, -1})
        \rightarrow    \\
     &  \Sigma^{2\tilde{V}_- \oplus 2V_{2,-} \oplus 2 (h_2 V_{1,-})} 
        \Sigma^{ V_{\lambda}^{0}(A_{0,0}, g, \bP) }
        I_{\lambda}^{\mu}(W^{i_1,i_2}_{n_1+1,n_2+1};A_{0,0}) = \\
     & \Sigma^{ 2 \tilde{\tilde{V}}_- }
        \Sigma^{ V_{\lambda}^{0}(A_{0,0}, g, \bP) }
        I_{\lambda}^{\mu}(W^{i_1,i_2}_{n_1+1, n_2+1};A_{0,0}).
   \end{split}
\end{equation}
Here we have used the fact that
\[
        2\tilde{V}_- \oplus 2V_{2,-} \oplus 2 (h_2 V_{1,-}) = 2\tilde{\tilde{V}}_-.
\]
Similarly taking the suspension of (\ref{eq map f_1 f_2}) by $\tilde{V}_{-}'$, we get
\begin{equation} \label{eq tilde V f1 f2}
       \begin{split}
            &  \Sigma^{ 2\tilde{V}_-' }   
            \hat{f}_1  \circ \hat{f}_2':  \\
            & \Sigma^{2\tilde{V}_-' \oplus  2V_{1, +} \oplus 2 (h_1 V_{2,+}) } 
              \Sigma^{ V_{\lambda}^0 (A_{-1,-1}, g, \bP)}
                 I_{\lambda}^{\mu}(W^{i_1,i_2}_{n_1+1,n_2+1};A_{-1, -1})
                 \rightarrow      \\
            &   \Sigma^{ 2  \tilde{ \tilde{V} }_-  } 
                \Sigma^{ V_{\lambda}^{0}(A_{0,0}, g, \bP) }
                I_{\lambda}^{\mu}(W^{i_1,i_2}_{n_1+1,n_2+1};A_{0,0}).
    \end{split}
\end{equation}

We will see that we can continuously deform (\ref{eq tilde V f2 f1}) to (\ref{eq tilde V f1 f2}).  Let $\gamma_0, \gamma_1$ be paths in $\cH_g^1(Y)$ from $-h_1-h_2$ to $0$ defined by
\[
  \begin{split}
   &  \gamma_0(t) = 
        \left\{
          \begin{array}{ll}
            (2t-1) h_1 - h_2 & \text{if $0 \leq t \leq \frac{1}{2}$,} \\
            (2t-2) h_2       & \text{if $\frac{1}{2} \leq t \leq 1$,  }
          \end{array}
        \right.                 \\
   & \gamma_1(t) =
        \left\{
           \begin{array}{ll}
             -h_1  + (2t-1) h_2 & \text{if $0 \leq t \leq \frac{1}{2}$,  }   \\
             (2t-2) h_1 & \text{if $\frac{1}{2} \leq t \leq 1$.}
           \end{array}
        \right.
  \end{split}
\]
Let $\Gamma:[0,1]^2 \rightarrow \cH^1_g(Y)$ be a homotopy from $\gamma_0$ to $\gamma_1$ defined by
\[
     \Gamma(u, t) = (1-u) \gamma_0(t) + t_1\gamma_1(t).
\]
There is $u_{1} \in (0, 1)$ such that for $u \in (0, u_1)$ the curve $\Gamma(u,\cdot)$ is not through $s_1(i)h_1 + s_2(1) h_2$ for $i \in \{ 1, \cdots, \ell_1-1 \}$ or $s_1(\ell_1-1) h_1 +  s_2(j) h_2$ for $j \in \{ 1, \cdots, \ell_2-1 \}$, and the curve $\Gamma(u_1,\cdot)$ is through the point $s_{1}(\ell_1-1) h_1 + s_2(1)h_2$. For each $u \in [0, u_{1})$, we can define a continuous map
\[
    \begin{split}
      \hat{h}_{u}:  
     & \Sigma^{2 (\tilde{V}_-(u) \oplus V_+(u))  }
        \Sigma^{ V_{\lambda}^0 (A_{-1,-1},\bP)}
                 I_{\lambda}^{\mu}(W^{i_1,i_2}_{n_1+1,n_2+1};A_{-1, -1}) \rightarrow    \\
     &   \Sigma^{  2\tilde{ \tilde{V} }_-  } 
                \Sigma^{ V_{\lambda}^{0}(A_{0,0}, \bP) }
                I_{\lambda}^{\mu}(W^{i_1,i_2}_{n_1+1, n_2+1};A_{0,0})
    \end{split}
\]
as before. This is continuous in $u$ since $\lambda(i,j)$ and $0$ are not an eigenvalue of $\tilde{D}_{s_1, s_2}$ for $(s_1,s_2) \in [s_1(i), s_1(i+1)] \times [s_2(j), s_2(j+1)]$.  Using the fact that there is a canonical isomorphism
\[
   \begin{split}
    & V_{1,+}(\ell_1-1,0) \oplus V_{2,+}(\ell_1, 1) \oplus
      V_{1,-}(\ell_1-1,1) \oplus V_{2,-}(\ell_1-1, 1)  
      \cong   \\
    & V_{2,+}(\ell_1-1,1) \oplus V_{1,+}(\ell_1-1,1) \oplus
      V_{1,-}(\ell_1-1,0) \oplus V_{2,-}(\ell_1, 1),
    \end{split}
\]
we can continuously extend $\hat{h}_{u}$ to $u \in [ u_{1}, u_{2})$, where $u_{2}$ is the next value such that $\Gamma(u_2,\cdot)$ is through $s_1(i)h_1 + s_2(j)h_2$ for some $i \in { 1, \dots, \ell_1-1}$, $j \in \{ 1, \dots, \ell_2-1 \}$. Repeating this discussion, we can define $\hat{h}_{t_1}$ for $t \in [0,1]$. The family $\{ \tilde{V}_-(u) \oplus V_+(u) \}_{u \in [0,1]}$ defines a vector bundle on $[0,1]$. Fix trivializations $\ft$ and $\tilde{\tilde{\ft}}$ of this bundle and $\tilde{ \tilde{V}}_-$. Then we get a homotopy from $\Sigma^{ 2 \tilde{V}_- }(\hat{f}_2 \circ h_2  \circ \hat{f_1} \circ h_1)$ to $\Sigma^{2 \tilde{V}'_-}(\hat{f}_1 \circ h_1 \circ \hat{f}_2 \circ h_2)$.

\subsection{Definition of $\SWF(Y, \fc, g, \bP)$: The case $b_1(Y) \geq 2$}

In this subsection, we will give the definition of $\SWF(Y, \fc, g, \bP)$ in the case where $b_1(Y) \geq 2$ and $c_1(\fc)$ is torsion, following \cite[Section 5]{KM ver 1}.  For simplicity, suppose that $b_1(Y) = 2$. In this case, $q_Y=0$ and we can find a spectral section $\bP$ for $\bD_{\fc}$.  Take a set $\{ h_1, h_2 \}$ of generators of $H^1(Y;\Z)$ and transverse double systems $(f_1, f_2)$ and $(f_1', f_2')$, where $(f_1, f_2)$ has the properties in (\ref{dfn transverse}) with respect to the action of $h_1$ and $(f_1', f_2')$ has the properties in (\ref{dfn transverse}) with respect to the action of $h_2$. Suppose that $N = N' = 2$ for simplicity, where $N, N'$ are the numbers of $V_{i}$, $V_{i}'$ which intersect with $U_n, U_n'$ as in Section \ref{ss transverse}. As before we have an object
$J(W^{i_1,i_2}_{n_1, n_2}) = J(W^{i_1, i_2}_{n_1, n_2};A_0,g, \bP)$ of $\fC$ for $i_1, i_2 \in \{ 1, 2 \}, n_1, n_2 \in \Z$. We also have the morphisms defined by using the flow and $\ff_j$:
\[
   \begin{split}
     & k^{i_2}:J(W^{1,i_2}_{n_1, n_2}) 
        \rightarrow 
             \Sigma \big( J(W^{2, i_2}_{n_1-1, n_2}) \vee 
                J(W^{2, i_2}_{n_1, n_2}) \big)
       \stackrel{ \ff_1 \vee \id  }{ \rightarrow }
            \Sigma J(W^{2, i_2}_{n_1, n_2}),  \\
    & l^{i_1}:J(W^{i_1,1}_{n_1, n_2}) 
        \rightarrow 
             \Sigma \big( J(W^{i_1, 2}_{n_1, n_2-1}) \vee 
                J(W^{i_1, 2}_{n_1, n_2}) \big)
       \stackrel{ \ff_2 \vee \id }{ \rightarrow }
            \Sigma J(W^{i_1, 2}_{n_1, n_2}).
   \end{split}
\]
We have the following diagram:
\[
   \xymatrix{
        J(W^{1,1}_{n_1, n_2}) \ar[r]^{k^1} \ar[d]_{l^1} & 
            \Sigma J(W^{2, 1}_{n_1,n_2}) \ar[d]^{\Sigma l^2}  \\
        \Sigma J(W^{1,2}_{n_1, n_2}) \ar[r]_{\Sigma k^2} & \Sigma^2 J(W^{2, 2}_{n_1,n_2})
   }
\]
We can see that the above diagram is commutative up to homotopy. Hence we have a morphism
\[
       L:\Sigma^{-1}C(k^1) \rightarrow C( \Sigma k^2).
\]
We define 
\[
       \SWF(Y,\fc, g, \bP) = \Sigma^{-1} C(L) \in \Ob(\fC).
\]
More precisely the definition is as follows. Let $\hat{k}^1, \hat{k}^2, \hat{l}^1, \hat{l}^2$ be the continuous maps which represent $k^1, k^2, l^1, l^2$ respectively, induced by choices of $T, \Delta$, $\lambda_j, \mu_j$. 
We consider the following diagram:
\begin{equation} \label{eq k l commu}
    \xymatrix{
      \Sigma^{2(V_+ \oplus V_+')}\Sigma^{ V_{\lambda}^0 } 
      I_{\lambda}^{\mu}(W_{n_1,n_2}^{1,1})
          \ar[r]^{\hat{k}^1} \ar[d]_{ \hat{l}^1 } &  
      \Sigma^{2( V_-  \oplus V_+' )} \Sigma^{ V_{\lambda}^0  } 
      \Sigma I_{\lambda}^{\mu}(W^{2,1}_{n_1,n_2})
          \ar[d]^{ \Sigma \hat{l}^2 } \\
      \Sigma^{2(V_+ \oplus V_-') }  \Sigma^{V_{\lambda}^0} 
      \Sigma I_{\lambda}^{\mu}(W^{1,2}_{n_1,n_2})
          \ar[r]_{ \Sigma \hat{k}^2 }
      &
      \Sigma^{2 (V_- \oplus V_-')  } \Sigma^{ V_{\lambda}^0 }
      \Sigma^2 I_{\lambda}^{\mu}(W^{2,2}_{n_1,n_2})
    }
\end{equation}

As in Section \ref{subsec fi fj}, if we choose  trivialization $\tilde{\tilde{\ft}}$ and $\ft$ of a vector space and a vector bundle on $[0,1]$, we get a homotopy from $\Sigma^{ 2\tilde{V}_- } (\Sigma \hat{l}^2 \circ \hat{k}^1)$ to $\Sigma^{ 2\tilde{V}_-'} (\Sigma \hat{k}^2 \circ \hat{l}^1)$. Here $\tilde{V}_-, \tilde{V}_-'$ are suitable vector spaces. Hence we have an induced continuous map
\[
     \hat{L}:  C( \Sigma^{2\tilde{V}_-}  \hat{k}^1)
              \rightarrow
               C( \Sigma^{2 \tilde{V}_-'}  \Sigma \hat{k}^2  ).
\]

\begin{dfn}
We define 
\[
   \begin{split}
   &  \SWF(Y, \fc, g, \bP) =  \\
   &  \SWF(Y, \fc, g, \bP; 
            A_0, \bn,  \Delta,  \{ \lambda(i,j), \mu(i,j) \}, 
              \ft_{\pm}, \ft, \tilde{\tilde{\ft}}, F) = \\
   &  ( C( \hat{L} ),  
      2\dim_{\R} (\tilde{V}_- \oplus V_+ \oplus V_+' \oplus V_{\lambda}^0)_{\R} + 2, 
    2 \dim_{\C} (\tilde{V}_- \oplus V_+ \oplus V_+' \oplus V_{\lambda}^0)_{\C} )  \\
   & \in \Ob (\fC).
   \end{split}
\]
Here $\bn = (n_1, n_2)$,  $F=( (f_1, f_2), (f_1', f_2') )$.
\end{dfn}

There is another way to define $\SWF(Y, \fc, g, \bP)$. The commutativity of the diagram (\ref{eq k l commu}) gives a continuous map
\[
        \hat{K}:
          C( \Sigma^{ \tilde{V}_-'' } \hat{l}_1) \rightarrow
          C( \Sigma^{ \tilde{V}_-'''} \Sigma \hat{l}_2).
\]
We define 
\[  
   \begin{split}
     &\SWF(Y, \fc, g, \bP) = \\
     & ( C( \hat{K} ), 2 \dim_{\R} ( \tilde{V}_-'' \oplus V_+ \oplus V_+'  )_{\R} + 2,  
         2 \dim_{\C} ( \tilde{V}_-'' \oplus V_+ \oplus V_+'  )_{\C}  ) \in \fC.
   \end{split}
\]
We can easily prove that this object is canonically isomorphic to the original one.

We assumed that $b_1(Y) = 2$ and $N = N' = 2$. However we can easily generalize this definition to any case, provided that $q_Y$ is trivial. In the case $b_1(Y) \geq 2$ we need trivializations $\ft$ of vector bundles on cubes $[0,1] \times \cdots \times [0,1]$. 
As in the case $b_1(Y) = 1$, for each submodule $H$ of $H^1(Y;\Z)$ of rank $b_1(Y)$, we can define a variant $\SWF(Y,\fc, H, g, \bP)$.

\begin{dfn}
Suppose that $q_Y = 0$.
Let $H$ be a submodule of $H^1(Y;\Z)$ of rank $b_1(Y)$ and take a set  $\{ h_1', \dots, h_{b}' \}$ of generators of $H$.  We denote by $\SWF(Y, \fc, H, g, \bP)$ the object of $\fC$ obtained by replacing $h_1, \dots, h_{b}$ with $h_1', \dots, h_{b}'$ in the construction of $\SWF(Y, \fc, g, \bP)$.
\end{dfn}

\begin{prop}  \label{prop indep 2}
The object $\SWF(Y, \fc, H, g, \bP)$ is independent of the choices of $A_0$, $\bn$, $\Delta$, $\lambda(i,j)$, $\mu(i,j)$, $\ft_{\pm}$, $\ft$, $\tilde{\tilde{\ft}}$ and $F$ up to canonical isomorphism in $\fC$.
\end{prop}


\subsection{Proof of Proposition \ref{prop indep 1} and Proposition \ref{prop indep 2}}
\label{subsec proof of prop indep}

\subsubsection{Independence from $\Delta$, $\{ \lambda_j, \mu_j \}_j$, $\ft_{\pm}$, $\ft$, $\tilde{\tilde{\ft}}$  } \label{subsec ind ft}
To simplify notation, we suppose that $b_1(Y) = 1$, $N = 2$, $H = H^1(Y;\Z)$. The proof for the general case is similar. 

Fix $\Delta, \lambda_j, \mu_j, \ft_{\pm}$.  Take another trivializations $\ft_{\pm}'$ of $V_{\pm}$. (Since $b_1(Y) = 1$, we do not need to take trivializations $\ft$ of vector bundles on $[0,1] \times \cdots \times [0,1]$ and $\tilde{\tilde{\ft}}$ of the vector space $\tilde{\tilde{V}}$. The proof of  independence from $\ft$ and $\tilde{\tilde{\ft}}$ is similar to that of independence from $\ft_{\pm}$.) We get another continuous map $\hat{k}'$ which represents the morphism $k$.
Then we have the following diagram:
\[
     \xymatrix{
       \Sigma^{2(\R^{d} \oplus \C^{d'})} \Sigma^{V_{\lambda}^0} 
       I_{\lambda}^{\mu}(W^1_n)
       \ar[r]^{\hat{k}} \ar[d]_{ (2\ft_+') \circ (2\ft_{+}^{-1})  } &
           \Sigma^{2(\R^{d} \oplus \C^{d'})} \Sigma^{V_{\lambda}^0}
           \Sigma I_{\lambda}^{\mu}(W^2_n)
           \ar[d]^{ (2\ft_-') \circ (2\ft_-^{-1})  }   \\
       \Sigma^{2(\R^{d} \oplus \C^{d'})} \Sigma^{V_{\lambda}^0} 
       I_{\lambda}^{\mu}(W^1_n) \ar[r]_{\hat{k}'} &
           \Sigma^{2(\R^{d} \oplus \C^{d'})} \Sigma^{V_{\lambda}^0}
           \Sigma I_{\lambda}^{\mu}(W^2_n)
     }
\]
This diagram is strictly commutative. Hence we get a homeomorphism
\[
     C(\hat{k}) \stackrel{\cong}{\rightarrow} C(\hat{k}').
\]
Hence we obtain an isomorphism
\[
          \SWF(Y, \fc, g, \bP;\ft_{\pm}) \stackrel{\cong}{\rightarrow}
          \SWF(Y, \fc, g, \bP;\ft_{\pm}')
\]
as required.

\begin{rem}
Since $\pi_0(O(d)) = \Z_2, \pi_0(U(d')) = 0$, we can take a homotopy from $2\ft_{\pm}$ to $2\ft_{\pm}'$. Using this homotopy, we get an isomorphism from $\SWF(Y, \fc, g, \bP;\ft_{\pm})$ to $\SWF(Y, \fc, g, \bP;\ft_{\pm}')$. However this isomorphism is not canonical, since $\pi_1(U(d')) = \Z$ and the homotopy from $2\ft_{\pm}$ to $2\ft_{\pm}'$ is not unique up to homotopy.
\end{rem}

Take another $\Delta', \lambda_j', \mu_j', \ft_{\pm}'$. We get another continuous map $\hat{k}'$.
It is sufficient to consider the case $\Delta \subset \Delta'$. By Lemma \ref{lem indep lambda mu}, we may suppose that $V_{\pm}' = V_{\pm} \oplus V_{\pm}''$ for some vector space $V_{\pm}''$ coming from $\Delta' \backslash \Delta$. Since we have proved the independence from the trivializations, we may suppose that $\ft_{\pm}' = \ft_{\pm} \oplus \ft_{\pm}''$ for some trivialization $\ft''_{\pm}$ of $V_{\pm}''$. Hence we have a canonical isomorphism
\begin{equation} \label{eq hat k cong}
     C(\hat{k}') = \Sigma^{2V_{+}'' - 2V_{-}''} C(\hat{k}) \cong C(\hat{k})
\end{equation}
in $\fC$. Here we have used the trivializations
\[
       V_+'' \stackrel{\ft_{+}''}{\rightarrow} \R^{d''_{\R}} \oplus \C^{d''_{\C}} 
        \stackrel{\ft_-''}{\leftarrow} V_{-}''.
\]
The isomorphism (\ref{eq hat k cong}) is independent of $\ft''_{\pm}$ since $\pi_0(O(d''_{\R})) = \Z_2, \pi_0(U(d''_{\C})) = 0$.

We can see the following diagram is commutative:
\[
     \xymatrix{
        \SWF(Y,\fc, g, \bP; \fd) \ar[rd]_{\cong} \ar[rr]^{\cong}  & &
        \SWF(Y,\fc, g, \bP; \fd'')  \\
         & \SWF(Y,\fc, g, \bP;\fd') \ar[ru]_{\cong} &
     }
\]
Here $\fd =( \Delta, \lambda_j, \mu_j, \ft_{\pm})$, $\fd' =( \Delta', \lambda_j', \mu_j', \ft_{\pm}')$, $\fd =( \Delta', \lambda_j', \mu_j', \ft_{\pm}')$.

If $b_2(Y) \geq 2$, we need to take a trivialization $\ft$ of a vector bundle on a cube and a trivialization $\tilde{\tilde{\ft}}$ of a vector space $\tilde{\tilde{V}}$. The proof of the independence from $\ft$ and $\tilde{\tilde{\ft}}$ is similar.

\subsubsection{Independence from $A_0$} \label{subsec ind A0}

Let $A_0'$ be another flat connection. 
As in Section \ref{subsec Iso f}, we can prove that there is a canonical isomorphism
\[
     \varphi: J(W^i_n;A_0, g, \bP)  \stackrel{\cong}{\rightarrow} J(W^i_n;A_0', g, \bP).
\]
We consider the following diagram:
\[
       \xymatrix{
    J(W^i_n;A_0, g, \bP) 
        \ar[r]^{\cong }_{\varphi} \ar[d]_{k}  & 
    J(W^i_n;A_0', \bP) \ar[d]^{k'} \\
    \Sigma J(W^{i+1}_n;A_0, g, \bP) 
         \ar[r]^{\cong }_{\varphi} &  
    \Sigma J(W^{i+1}_n;A_0', g, \bP)  &
       }
\]
This diagram is commutative up to canonical homotopy. More precisely, as in Section \ref{subsec fi fj}, we can prove that after taking a trivialization $\ft$ of a vector bundle $W$ on $[0,1]$ we can define a homotopy $H$ between $\Sigma^{2 \tilde{V}_-} (\hat{k}' \circ \hat{\varphi})$ and $\Sigma^{2 \tilde{V}'} (\hat{\varphi} \circ \hat{k})$, using the trivialization $2\ft$ of $2W = W \oplus W$.  Here $\tilde{V}_-$ and $\tilde{V}_-'$ are suitable finite dimensional vector spaces, and $\hat{k}, \hat{k}', \hat{\varphi}$ are continuous maps which represent $k, k', \varphi$ induced by choices of a partition $\Delta$ of $[-1,0]$ and positive large numbers $T, -\lambda_j, \mu_j$ .  The homotopy $H$ and $\varphi$ induce an isomorphism
\[
      \Sigma^{ 2\tilde{V}_- } C( \hat{k} ) \stackrel{\cong}{\rightarrow}
      \Sigma^{ 2\tilde{V}_-'} C( \hat{k}').
\]
Fix trivializations $\tilde{\ft}_-$ and $\tilde{\ft}_-'$ of $\tilde{V}_-$ and $\tilde{V}_-'$. Then, taking desuspensions, we get an isomorphism
\[
        \SWF(Y, \fc, g, \bP; A_0) \stackrel{\cong}{\rightarrow} \SWF(Y, \fc, g, \bP; A_0').
\]
We can prove that this isomorphism is independent of $\ft, \tilde{\ft}_-, \tilde{\ft}_-'$ as in the previous subsection. We can also prove that the following diagram is commutative:
\[
     \xymatrix{
        \SWF(Y,\fc, g, \bP; A_0) \ar[rd]_{\cong} \ar[rr]^{\cong}  & &
        \SWF(Y,\fc, g, \bP; A_0'')  \\
         & \SWF(Y,\fc, g, \bP;A_0') \ar[ru]_{\cong} &
     }
\]

\subsubsection{Independence from $F$}
\label{subsec ind trans}

We will prove the independence from the choice of transverse double systems. First we suppose that $b_1(Y) = 1$.
Take two transverse double systems $(f_1, f_2)$ and $(\tilde{f}_1, \tilde{f}_2)$. We want to show that $\SWF(Y, \fc, g, \bP; f_1, f_2)$ and $\SWF(Y, \fc, g, \bP; \tilde{f}_1, \tilde{f}_2)$ are canonically isomorphic in $\fC$. Write $\widetilde{W}_n^i$ for the subset $W_n^i$ of $Str(R)$ associated with $(\tilde{f}_1, \tilde{f}_2)$. It is sufficient to consider the case where $f_2 = \tilde{f}_2$.
For simplicity, suppose that $N = 2, \tilde{N} = 3$, where $N, \tilde{N}$ are the numbers of $V_{i}, \tilde{V}_i$ which intersect $U_n, \tilde{U}_n$ respectively, as in Section \ref{ss transverse}. By renumbering if necessary, we have  $W^1_n \cup W^2_n = \widetilde{W}^1_n \cup \widetilde{W}^2_n \cup \widetilde{W}^3_n$.
First we suppose that $\widetilde{W}^3_n \subset W^2_n$.
Then we can write
\[
      W^1_n = \widetilde{W}_n^1 \cup Z_n^1, \
      W^2_n = \widetilde{W}_n^3 \cup Z_n^2, \
      \widetilde{W}_n^2 = Z^1_n \cup Z_n^2.
\]
We have canonical isomorphisms
\begin{equation} \label{eq J cano iso}
    \begin{split}
 & \Sigma J(W^1_n) \cong C \big( l_1:J(\widetilde{W}_n^1) \rightarrow \Sigma J(Z^1_n) \big), \\
 & \Sigma J(W^2_n) \cong C \big( l_2:J(Z^2_n ) \rightarrow \Sigma J(\widetilde{W}^3_n) \big), \\
 & \Sigma J( \widetilde{W}^2_n ) \cong C \big( l_3:J(Z^1_n) \rightarrow \Sigma J(W^2_n)  \big).
    \end{split}
\end{equation}
By definition, we have
\[
          \SWF(Y,\fc, g, \bP; f_1, f_2) = \Sigma^{-1}C(k)
\]
where
\[
         k:J(W^1_n) \rightarrow \Sigma (J(W^2_{n-1}) \vee J(W^2_n))
                    \rightarrow \Sigma J(W^2_n).
\]
On the other hand
\[
        \SWF(Y, \fc, g, \bP; \tilde{f}_1, \tilde{f}_2) 
          = \Sigma^{-1} C(\tilde{K})
\]
where $\tilde{K}$ is the morphism $\Sigma^{-1} C(\tilde{k}^1) \rightarrow \Sigma J( \widetilde{W}^i_n)$.
We want to show that there is a canonical isomorphism between $C(k)$ and $C(\tilde{K})$. We have the exact triangles 
\begin{equation} \label{eq exact C(K)}
   \begin{split}
    &  \Sigma^{-1}C(\tilde{k}^1) \stackrel{\tilde{K}}{\longrightarrow} 
       \Sigma J(\widetilde{W}^3_n) \longrightarrow
       C( \tilde{K}),  \\
    &  J(W^1_n) \stackrel{k}{\longrightarrow} \Sigma J(W^2_n) \longrightarrow C(k).
   \end{split}
\end{equation}
We also have the exact triangle
\begin{equation} \label{eq exact W^2_n}
     J(Z^2_n) \longrightarrow \Sigma J(\widetilde{W}^3_n) 
              \longrightarrow \Sigma J(W^2_n).
\end{equation}
Hence we have the following diagram:
\[
   \xymatrix{
     \Sigma^{-1} C( \tilde{k}^1) \ar[r]^{\tilde{K}} & 
                   \Sigma J( \widetilde{W}_n^3) \ar[r] \ar[d]   & C(\tilde{K})  \\
     J(W^1_n) \ar[r]^{k} & \Sigma J(W^2_n) \ar[r] \ar[d] & C(k)    \\
                         & \Sigma J(Z^2_n)  &
   }
\]
Next we show that the following is exact:
\begin{equation} \label{eq exact C(K^1)}
        \Sigma J(Z^2_n) \longrightarrow C( \tilde{k}^1 ) \longrightarrow \Sigma J(W^1_n).
\end{equation}
Here the morphism $\Sigma J(Z^2_n) \rightarrow C( \tilde{k}^1 )$ is defined as follows.
We  can write
\[
       C( \tilde{k}^1) =
         C J(\widetilde{W}^1_n) \cup_{\tilde{k}^1} \Sigma J( \widetilde{W}^2_n ).
\]
Moreover there is a canonical isomorphism
\[
        \Sigma J( \widetilde{W}^2_n) \cong CJ(Z^1_n) \cup_{l_2} \Sigma J(Z^2_n),
\] 
The morphism $\Sigma J(Z^2_n) \rightarrow C( \tilde{k}^1 )$ is given by
\[
          \Sigma J(Z^2_n) 
            \rightarrow CJ(Z^1_n) \cup_{l_2} \Sigma J(Z^2_n) 
            \cong \Sigma J(\widetilde{W}^2_n)
            \rightarrow C(\tilde{k}^1).
\]
Collapsing $\Sigma J(Z^2_n)$ into one point, we get
\[
      C(\tilde{k}^1) / \Sigma J(Z^2_n) \cong \Sigma J(W^1_n)
\]
Note that the composition
\[
     J(\widetilde{W}^1_n) \stackrel{ \tilde{k}^1}{\longrightarrow}
     \Sigma J( \widetilde{W}^2_n ) \longrightarrow
     \Sigma J( \widetilde{W}^2_n )/
            \Sigma J(Z^2_n)
      = \Sigma J(Z^1_n)
\]
is the usual morphism $J(\widetilde{W}^1_n) \rightarrow \Sigma J(Z^1_n)$. 
Therefore the sequence (\ref{eq exact C(K^1)}) is exact.

From (\ref{eq exact C(K)}), (\ref{eq exact W^2_n}) and (\ref{eq exact C(K^1)}), we get the following diagram:
\[
     \xymatrix{
         \Sigma^{-1} C(\tilde{k}^1) \ar[r] \ar[d] & 
         \Sigma J( \widetilde{W}^3_n) \ar[r] \ar[d] &
         C(\tilde{K})  \\
         J(W^1_n) \ar[r] \ar[d] & \Sigma J(W^2_n) \ar[r] \ar[d] & C(k) \ar[d]    \\
         \Sigma J(Z^2_n) \ar[r]^{\id} & \Sigma J(Z^2_n) \ar[r] & {*}
     }
\]
We can see that this diagram is commutative up to canonical homotopy. More precisely, as in the previous subsection, to see the homotopy commutativity of the diagram, we need to fix some trivializations of vector spaces and a vector bundle on $[0,1]$. We omit the details.  The homotopy commutativity of this diagram induces the canonical isomorphism
\[
        C( \tilde{K}) \cong C(k)
\]
as required.

Next we consider the case $\widetilde{W}^3_n \not\subset W^2_n$.
In this case we can write
\[
   \begin{split}
      & W^1_n = \widetilde{W}^1_n \cup Z^1_n \cup Z^2_n, \
        W^2_n = Z^3_n \cup Z^4_n,       \\
      & \widetilde{W}^2_n = Z^1_n \cup Z^3_n,       \
        \widetilde{W}^3_n = Z^2_n \cup Z^4_n.
   \end{split}
\]
Put 
\[
     \widetilde{W}^1_n \ \!' = \widetilde{W}^1_n, \ 
     \widetilde{W}^2_n \ \!' = Z^1_n \cup Z^2_n \cup Z^3_n, \ 
     \widetilde{W}^3_n \ \!' = Z^4_n.
\]
We can define the morphisms:
\[
   \begin{split}
     & \tilde{k}^1 \ \! ': 
        J(\widetilde{W}_n^{1} \ \! ' ) \rightarrow
        \Sigma J(\widetilde{W}_{n-1}^2 \ \! ') \vee \Sigma J(\widetilde{W}_n^2 \ \! ')
              \rightarrow
                    \Sigma J(\widetilde{W}_n^2 \ \! '), \\
     & \tilde{K}': \Sigma^{-1} C ( \tilde{k}^1 \ \! ' ) \rightarrow 
                    \Sigma J(\widetilde{W}^3_n \ \!').
    \end{split}
\]
It is easy to see that $C( \tilde{K})$ and $C( \tilde{K}')$ are canonically isomorphic to each other.
Since $\widetilde{W}_n^3 \ \! ' \subset W^2_n$, we can prove that $C(k)$ is canonically isomorphic to $C( \tilde{K}' )$ as before. Therefore $C(k)$ is canonically isomorphic to $C( \tilde{K})$.

Although we assumed that $N = 2, \tilde{N} = 3$, we can generalize our discussion to any case.

Suppose that $b_1(Y) = 2$. Let $\{ h_1, h_2 \}$ and $\{ \tilde{h}_1, \tilde{h}_2 \}$ be sets of generators of $H^1(Y;\Z)$. Take transverse double systems $F = ( (f_1, f_2), (f_1', f_2'))$, $\tilde{F} = ((\tilde{f}_1, \tilde{f}_2), (\tilde{f_1}', \tilde{f}_2'))$ with respect to $\{ h_1, h_2 \}$, $\{ \tilde{h}_1, \tilde{h}_2 \}$. We want to show that $\SWF(Y, \fc, g, \bP; F)$ is isomorphic to $\SWF(Y,\fc, \bP; \tilde{F})$. It is sufficient to consider the case $h_1 = \tilde{h}_1$. As in the case where $b_1(Y) = 1$, we can show that there are canonical isomorphisms
\[
         C( k^1) \stackrel{\cong}{\rightarrow}
         C( \tilde{k}^1), \quad
         C(k^2) \stackrel{\cong}{\rightarrow}
         C (\tilde{k}^2).
\]
Moreover we can see that the following diagram is commutative up to canonical homotopy:
\[
     \xymatrix{
        C(k^1) \ar[r]^{\cong} \ar[d]_{L} & C(\tilde{k}^1) \ar[d]^{ \tilde{L} }  \\
        C(k^2) \ar[r]_{\cong} & C(\tilde{k}^2)
              }
\]
Hence we get an isomorphism $\SWF(Y, \fc, g, \bP; F) \cong \SWF(Y, \fc, g, \bP; \tilde{F})$.
We can see that the following diagram is commutative:
\[
     \xymatrix{
        \SWF(Y,\fc, g, \bP; F) \ar[rd]_{\cong} \ar[rr]^{\cong}  & &
        \SWF(Y,\fc, g, \bP; \tilde{\tilde{F}})  \\
         & \SWF(Y,\fc, g, \bP; \tilde{F}) \ar[ru]_{\cong} &
     }
\]


\subsubsection{Independence from $\bn$}

Assume that $b_1(Y) = 2$ for simplicity. 
We will see that $\ff_1$ induces an isomorphism
\[
     \Phi_{\ff_1}:\SWF(Y,\fc, g, \bP; n_1, n_2)
                  \stackrel{\cong}{\rightarrow}
                  \SWF(Y, \fc, g, \bP; n_1+1, n_2).
\]
We have the isomorphism
\[
    \ff_1:J(W_{n_1, n_2}^{i_1, i_2}) \stackrel{\cong}{\rightarrow}
              J(W_{n_1+1, n_2}^{i_1, i_2}).
\]
This induces a homotopy equivalence
\[
       \hat{\varphi}_{\ff_1}:
        C( \hat{k}^1;n_1,n_2) \rightarrow
        C( \hat{k}^1;n_1+1,n_2).
\]
Consider the following diagram:
\[
      \xymatrix{
        C( \hat{k}^1;n_1,n_2) 
             \ar[r]^{ \hat{\varphi}_{\ff_1}} \ar[d]_{\hat{L}}  & 
             C( \hat{k}^1;n_1+1,n_2) \ar[d]^{\hat{L}}     \\
        C( \hat{k}^2;n_1,n_2) 
             \ar[r]_{ \hat{\varphi}_{\ff_1} } &
        C( \hat{k}^2;n_1+1,n_2)
      } 
\]
As in Section \ref{subsec fi fj}, we can construct a homotopy from $\Sigma^{ 2V_- } (\hat{\varphi}_{\ff_1} \circ \hat{L})$ to $\Sigma^{2V_-'} (\hat{L} \circ \hat{\varphi}_{\ff_1})$ if we choose a trivialization $\ft$ of a vector bundle on $[0,1]$. Here $V_-, V_-'$ are suitable vector spaces.
This homotopy induces an homotopy equivalence
\begin{equation} \label{eq iso W L}
                 C(\Sigma^{2V_-} \hat{L};n_1,n_2 ) \rightarrow 
                 C(\Sigma^{2V_-'} \hat{L};n_1+1, n_2).
\end{equation}
On the other hand, fix trivializations 
\[
            V_- \stackrel{\ft_-}{\rightarrow} \R^{d_1} \oplus \C^{d_2}
                \stackrel{\ft_-'}{\leftarrow} V_-'
\]
Then in the category $\fC$ we get isomorphisms induced by $\ft_-, \ft'_-$:
\begin{equation} \label{eq iso L 1}
   \begin{split}
      & C(\hat{L}; n_1,n_2) \stackrel{\cong}{\rightarrow}
        (C(\Sigma^{2V_-} \hat{L},n_1, n_2), d_1, d_2),  \\
      & C(\hat{L}; n_1+1,n_2) \stackrel{\cong}{\rightarrow}
        ( C(\Sigma^{2V_-'} \hat{L}; n_1+1, n_2), d_1, d_2 )
   \end{split}
\end{equation}
Combining (\ref{eq iso W L}) and (\ref{eq iso L 1}), we get an isomorphism
\[
         \Phi_{\ff_1}:\SWF(Y, \fc, g, \bP;n_1,n_2) \stackrel{\cong}{\rightarrow}
                      \SWF(Y, \fc, g, \bP; n_1+1, n_2).
\]
Since $\pi_0(O(N)) = \Z_2, \pi_0(U(N)) = 0$, $\Phi_{\ff_1}$ is independent of $\ft, \ft_-, \ft_-'$.
The proof for the case $b_1(Y) \geq 3$ is similar.

We can prove the following digram is commutative:

\[
      \xymatrix{
     \SWF(Y,\fc, g, \bP;\bn) \ar[rr]^{\Phi_{\ff_j \circ \ff_i}} \ar[rd]_{\Phi_{\ff_i}} & 
                          & \SWF(Y, \fc, g, \bP, \bn'')   \\
         & \SWF(Y, \fc, g, \bP; \bn')  \ar[ru]_{ \Phi_{\ff_j} }
      }
\]

\section{Relative invariant}

\subsection{Stable cohomotopy version of Seiberg-Witten invariants for closed manifolds}

Before we construct the relative stable cohomotopy version of Seiberg-Witten invariants for 4-manifolds with boundary, we briefly review the construction of the invariant for closed 4-manifolds. See \cite{BF} for the detail.

Let $X$ be a closed, oriented 4-manifold and choose a Riemannian metric $\hat{g}$ on $X$.
Take a spin-c structure $\hat{\fc}$ of $X$ and fix  a connection $\hat{A}_0$ of $\det \hat{\fc}$. We denote by $\bbS^+$, $\bbS^-$ the spinor bundles on $X$ associated with $\hat{\fc}$ and by $\hat{\rho}:T^* X \rightarrow \Hom(\bbS^+, \bbS^-)$ the Clifford multiplication. Let $\Omega_{\hat{g}}^1(X)$ be the image of $\hat{d}^*:\Omega^2(X) \rightarrow \Omega^1(X)$ and $\Omega^+_{\hat{g}}(X)$ be the space of self-dual 2-forms on $X$.
Put
\[
    \begin{split}
    & \cE(X) = 
       L^2_{k+1}(\sqrt{-1} \Omega_{\hat{g}}^1(X) \oplus \Gamma(\bbS^+)  ),
         \\
    & \cF(X) = 
       L^2_{k}( \sqrt{-1} \Omega^+_{\hat{g}}(X) \oplus \Gamma(\bbS^-)  ).
    \end{split}
\]
We define $U(1)$-actions on $\cE(X)$ and $\cF(X)$ by multiplications on $\bbS^+$ and $\bbS^-$. The Seiberg-Witten map is defined by
\[
    \begin{array}{rcl}
      SW:\cE(X)     & \rightarrow & \cF(X) \\
      (\hat{a}, \hat{\phi}) & \mapsto     & 
       (F_{\hat{A}_0+ \hat{a}}^+ + q( \hat{\phi}), 
        D_{\hat{A}_0 + \hat{a}} \hat{\phi}).
    \end{array}
\]
Here $F_{\hat{A}_0+\hat{a}}^+$ is the self-dual part of the curvature $F_{ \hat{A}_0+\hat{a}}$ and $q(\hat{\phi})$ is an endomorphism of $\bbS^+$ defined by $\hat{\phi} \otimes \hat{\phi}^* - \frac{1}{2} | \hat{\phi} |^2 \id$ which is considered to be a self-dual 2-form through an isomorphism $\Lambda^+ T^*X \cong \End(\bbS^+)$ induced by $\hat{\rho}$.

We take a finite dimensional approximation of the map $SW$ as follows. We can write $SW$ as $L + C$, where $L$ is the linear part of $SW$ and defined by $L(\hat{a}, \hat{\phi}) = (d^+ \hat{a}, D_{ \hat{A}_0} \hat{\phi})$, and $C(\hat{a}, \hat{\phi}) = ( F_{\hat{A}_0}^+ + q(\hat{\phi}), \rho(\hat{a}) \hat{\phi})$. 
Let $U$ be finite dimensional subspace of $\cF(X)$ such that $\Image L + U = \cF(X)$ and put $U' = L^{-1}(U)$. Then using the compactness of the Seiberg-Witten moduli space, we can show that the map
\[
      SW_U = \pr_U \circ SW|_{U'}:U' \rightarrow U
\]
extends a map
\[
       SW_U^+:(U')^+ \rightarrow U^+.
\]
Here $U^+$ and $(U')^+$ are the one point compactification of $U$ and $U'$ respectively.
Bauer and Furuta \cite{BF} showed that for sufficiently large $U$, the $U(1)$-equivariant homotopy class of $SW_U^+$ is stable (in a suitable sense.) Hence we get  an element $\psi_{X, \hat{\fc}}$ of a $U(1)$-equivariant stable cohomotopy group $\pi^{b^+(X)}_{U(1)}(*;\Ind D_{\hat{\fc}})$ and $\psi_{X, \hat{\fc}}$ is an invariant of $X$ which is independent of the choices of $\hat{g}$ and $U$. More precisely, in \cite{BF}, Bauer and Furuta constructed an element $\Psi_{X, \hat{\fc}}$ of a stable cohomotopy group $\pi^{b^+(X)}_{U(1)}(\Pic(X);\Ind D_{ \hat{\fc} })$ of the Thom space of the index bundle on $\Pic(X)$ of a family of Dirac operators. $\psi_{X, \hat{\fc}}$ is the restriction of $\Psi_{X, \hat{\fc}}$ to the fiber of the index bundle.

\subsection{Relative invariant}     \label{subsec relative inv}
We will define the relative invariant following \cite{Khandhawit, KM ver 1, KM ver 3, Mano b1=0}. Let $Y$ be a closed, oriented 3-manifold with $q_Y = 0$. Take a Riemannian metric $g$, a spin-c structure $\fc$ on $Y$ with $c_1(\fc)$ torsion and a spectral section $\bP = \{ P_h \}_{[h] \in \Pic(Y)}$ for the family $\bD_{\fc}$ of Dirac operators on $Y$. Let $X_1$ be a compact, oriented 4-manifold with $\partial X_1 = Y$. Fix a Riemannian metric $\hat{g}_1$ and a spin-c structure $\hat{\fc}_1$ on $X_1$ with $\hat{g}_1|_Y = g, \hat{\fc}_1|_{Y} = \fc$ and a connection $\hat{A}_1$ on $\det \hat{\fc}_1$ with $\hat{A}_1|_{Y} = A_0$, where $A_0$ is a fixed flat connection on $\det \fc$.
Put
\begin{gather*}
     \Omega_{\hat{g}_1}^1(X_1) =  \\
     \left\{  \ \hat{a}_1
       \ \left| 
       \begin{array}{l}
   \hat{a}_1 \in \sqrt{-1} \ker ( \hat{d}^* : \Omega^1(X_1) \rightarrow \Omega^0(X_1)  ), \\
   \hat{d}^*( i^* \hat{a}_1 ) = 0,  \\
   \int_{Y_j} \hat{a}_1(\nu) = 0 \ ( j = 1, \dots, r)
       \end{array}
     \right.
     \right\}.
\end{gather*}
Here $i$ is the inclusion $Y \hookrightarrow X_1$,  $\nu$ is the normal vector field on $Y$ and $Y_j$ is a connected components of $Y = Y_1 \coprod \cdots \coprod Y_r$. The boundary condition in the definition of $\Omega_{g_1}^1(X_1)$ was introduced by Khandhawit in \cite{Khandhawit} and is called the double Coulomb condition. Let $\cU_{X_1}$ be the orthogonal complement in $L^2_{k+1}(\Omega^1_{\hat{g}_1}(X_1))$ of the space $\cH^1(X_1)$ of harmonic 1-forms on $X_1$ satisfying the double Coulomb condition.
The Seiberg-Witten map $SW^{\mu}$ of $X_1$ is 
\[
  \begin{array}{rcl}
    SW^{\mu}: 
      \cU_{X_1} \oplus  L^2_{k+1}(\Gamma (\bbS^+)) &
      \rightarrow & 
     L^2_{k}( \Omega_{\hat{g}_1}^+(X_1) \oplus \Gamma(\bbS^-) ) \oplus V^{\mu} \\
    \hat{x}_1 = (\hat{a}_1, \hat{\phi}_1) & \mapsto & 
      ( sw( \hat{x}_1), p^{\mu}  (i^*\hat{x}_1) ),
  \end{array}
\]
where $V^{\mu}$ is the subspace of $V$ spanned by eigenvectors of $D_{A_0} + B_{0}^{\bP}$ with eigenvalues in $(-\infty, \mu]$ and
\[
     sw( \hat{x}_1) = 
       (F_{ \hat{A_1}+\hat{a}_1}^+ + q(\hat{\phi}_1), 
         D_{ \hat{A}_1 + \hat{a}_1} \hat{\phi}_1).
\]

Take small neighborhoods $N, N'$ of $Y$ in $X_1$ with $N \subset N'$ and a smooth function $\tau:X_1 \rightarrow [0,1]$ with $\tau = 1$ on $N$ and $\tau = 0$ on $X_1 \backslash N'$.
We write $\hat{L}_{\hat{A}_1, \bP}$ for the following operator:
\[
      \hat{L}_{ \hat{A}_1, \bP } = 
      \hat{d}^+ \oplus ( D_{\hat{A}_1} +  B_{0}^{\bP} \tau) \oplus p^{\mu} i^*.
\]
This operator is Fredholm.
We can take a finite dimensional subspace 
\[
   U_1 \subset
    L^2_{k}( \Omega^+(X_1) \oplus \Gamma(\bbS^-) )
\] 
such that $\Image \hat{L}_{\hat{A}_1, \bP}$ and $U_1 \oplus  V_{\lambda}^{\mu}$ are transverse for $\lambda \ll 0$.
Write $U_1'$ for the preimage of $U_1 \oplus V_{\lambda}^{\mu}$ by $\hat{L}_{\hat{A}_1, \bP}$.
We get a finite dimensional approximation
\[
     SW_{U_1, \lambda}^{\mu} = 
      \pr_{U_1 \oplus V_{\lambda}^{\mu}} \circ SW^{\mu}|_{U_1 '}: 
     U_1' \longrightarrow U_1 \oplus V_{\lambda}^{\mu}
\]
of the Seiberg-Witten map. 
We will show that this map defines a morphism
\begin{gather*}
    \psi_{X_1} =\psi_{X_1, \hat{\fc}_1, H, g, \bP} \\
      \in
        \{ (\Sigma^{- V_{\lambda}^{0}} (U_1')^+, 0), 
           U_1^+ \wedge \SWF(Y, \fc, H, g, \bP) \}^{U(1)} \\
      = \{ (\C^a)^+, \Sigma^{b^+(X)} \SWF(Y, \fc, H, g, \bP) \}^{U(1)}.
\end{gather*} 
in $\fC$. Here $a$ is the numerical index of the Dirac operator on $X_1$ and $H$ is a submodule of $H^1(Y;\Z)$ of rank $b_1(Y)$.

Assume that $b_1(Y) = 1$. (The general case is similar.)
Write $\gamma_{\lambda}^{\mu}$ the flow on $V_{\lambda}^{\mu}$ induced by $CSD$.
Take positive large numbers $R \gg R'_1 \gg 0$ and choose a transverse double system $(f_1, f_2)$ on $Str(R)$. For simplicity, we suppose that $N = 2$.
If $n$ is large, for any $U_1$, we have
\[
   \begin{split}
     & i^*(B(U_1', R'_1)) \subset W^{1}_{-n} \cup W^2_{-n} \cup \cdots \cup W^1_{n} \cup W^2_n.
   \end{split}
\]
Here $W^i_{k} \subset Str(R)$ is defined as in Section \ref{ss transverse}.
Put $\widetilde{W}:= W^{1}_{-n} \cup W^2_{-n} \cup \cdots \cup W^1_{n} \cup W^2_n$.

\begin{lem} \label{lem N L}
There is an index pair $(N, L)$ for $\Inv ( \widetilde{W} \cap V_{\lambda}^{\mu})$ such that
\begin{equation} \label{eq cond N L}
      N \subset \widetilde{W} \cap V_{\lambda}^{\mu},
         \quad
      p_{\lambda}^{\mu}( i^*(B(U_1', R'_1))  ) \subset N \backslash L.
\end{equation}

\end{lem}

\begin{proof}
Fix a large compact set $B$ in $\widetilde{W} \cap V_{\lambda}^{\mu}$, which is diffeomorphic to a closed ball of dimension $\dim_{\R} V_{\lambda}^{\mu}$, is an isolating neighborhood of $\Inv (\widetilde{W} \cap V_{\lambda}^{\mu})$ and includes $p_{\lambda}^{\mu}( i^* B(U_1', R_1') )$.
Let $\chi:\widetilde{W} \cap V_{\lambda}^{\mu} \rightarrow [0,1]$ be a smooth function such that
\[
   \begin{split}
      \chi^{-1}(0) &= B, \\
      \chi &= 1 \ 
         \text{on a neighborhood of $\partial ( \widetilde{W} \cap V_{\lambda}^{\mu} )$}.
   \end{split}
\]
Note that the flows $\gamma_{\lambda}^{\mu}$ and $\chi \gamma_{\lambda}^{\mu}$ have the same directions outside $B$. Hence $\widetilde{W} \cap V_{\lambda}^{\mu}$ is an isolating neighborhood of $\Ind(\widetilde{W} \cap V_{\lambda}^{\mu}; \chi \gamma_{\lambda}^{\mu})$ with respect to $\chi \gamma_{\lambda}^{\mu}$. 
Let $(N, L)$ be an index pair of $\Inv ( \widetilde{W} \cap V_{\lambda}; \chi \gamma_{\lambda}^{\mu} )$ with respect to the flow $\chi \gamma_{\lambda}^{\mu}$.  Then 
\[
       N \subset \widetilde{W} \cap V_{\lambda}^{\mu}, \
       p_{\lambda}^{\mu} i^* ( B(U_1', R_1' )) \subset B \subset  N \backslash L.
\]
The pair $(N,L)$ is also an index pair for $\Inv(\widetilde{W} \cap V_{\lambda}^{\mu};\gamma_{\lambda}^{\mu})$ with respect to the original flow $\gamma_{\lambda}^{\mu}$ since $\gamma_{\lambda}^{\mu}$ and $\chi \gamma_{\lambda}^{\mu}$ have the same directions outside $B$ as stated.
\end{proof}

Fix a regular index pair $(N, L)$ satisfying (\ref{eq cond N L}).
(See \cite[Definition5.1]{Sala} for the definition of a regular index pair. We can always find a regular index pair \cite[Remark 5.4]{Sala}.)
We will see that the following is well defined and continuous for large $-\lambda, \mu$, $U_1$, $T > 0$ and small $\epsilon > 0$:
\begin{gather} \label{eq rel}
      \widetilde{\psi}: (U_1')^+ \longrightarrow U_1^+ \wedge I_{\lambda}^{\mu}(\widetilde{W}) \\
      \widetilde{\psi}( \hat{x}_1 ) =   \nonumber \\
        \left\{
          \begin{array}{ll}
           ( \pr_{U_1}  sw(\hat{x}_1  ), y_1 \cdot T) &
              \text{if $\| \pr_{U_1}  sw(\hat{x}_1) \| < \epsilon$,
                    $y_1 \cdot [0, T] \subset 
                      N \backslash L$,}
                          \nonumber \\
             * & \text{otherwise}.
          \end{array}
        \right.
\end{gather}
Here $y_1 = p_{\lambda}^{\mu}  i^* \hat{x}_1$ and we think of $U_1^+$ and $(U_1')^+$ as $B(U_1, \epsilon)/S(U_1, \epsilon)$ and $B(U_1', R'_1)/S(U_1', R'_1)$ respectively.

\begin{lem} \label{lem rel inv T_0}
Let $(N, L)$ be a regular index pair of $\Inv( \widetilde{W} \cap V_{\lambda}^{\mu})$ satisfying (\ref{eq cond N L}).
Fix  large positive numbers $R \gg R_1' \gg 0$. There is $T_0 > 0$ independent of $U_1, \lambda, \mu, \epsilon$ such that if $T > T_0$ for large $U_1$,  $-\lambda, \mu, \gg 0$ and small $\epsilon > 0$, (\ref{eq rel}) is well defined and continuous.
\end{lem}

\begin{proof}
To prove the map (\ref{eq rel}) is well defined, we need to show that if $\hat{x}_1 \in U_1', \| \hat{x}_1 \| = R'_1$ and $\| \pr_{U_1}  sw(\hat{x}_1) \| < \epsilon$, then $y_1 \cdot [0, T] \not \subset N \backslash L$. Assume that the proposition is false. Then we have sequences $T_{\alpha}$, $-\lambda_{\alpha}, \mu_{\alpha},  \rightarrow \infty$, $U_{1, \alpha}$ with $\dim U_{1, \alpha} \rightarrow \infty$, $\epsilon_{\alpha} \rightarrow 0$, $\hat{x}_{1,\alpha} \in U'_{1,\alpha}$ with $\| \hat{x}_{1, \alpha} \| = R'_1$ such that $y_{1, \alpha} \cdot [0, T_{\alpha}] \subset N_{\alpha} \backslash L_{\alpha} \subset \widetilde{W}$. Here $(N_{\alpha}, L_{\alpha})$ is an index pair for $\Inv(\widetilde{W} \cap V_{\lambda_{\alpha}}^{\mu_{\alpha}})$ satisfying (\ref{eq cond N L}) and $y_{1, \alpha} =  p_{\lambda_{\alpha}}^{\mu_{\alpha}} i^*( \hat{x}_{1, \alpha} )$. 
The assumptions that $\| x_{1, \alpha} \| = R_1'$ and that $y_{1, \alpha} \cdot [0, T_{\alpha}] \subset N_{\alpha} \backslash L_{\alpha} \subset \widetilde{W}$ imply that the energy of $(\hat{x}_{1, \alpha}, \{ y_{1, \alpha} \cdot T \}_{0 \leq T \leq T_{\alpha}})$ is bounded by  a constant independent of $\alpha$.
By (a slightly different version of) Lemma 2 in \cite{Khandhawit}, we can find subsequences $\hat{x}_{1,\alpha'}$ conversing to a solution $\hat{x}_1$ to the Seiberg-Witten equations on $X_1$ with $\| \hat{x}_1 \| = R_1'$ and $y_{1 \alpha'}:[0,T_{\alpha'}] \rightarrow V_{\lambda_{\alpha'}}^{\mu_{\alpha'}}$ conversing to a finite energy trajectory $y_{1}:[0,\infty) \rightarrow V$ on each compact set in $[0, \infty)$, and we have $i^*(\hat{x}_1) = y_1(0)$. Since $R_1' \gg 0$, this is a contradiction to Corollary 2 in \cite{Khandhawit}.

Using the assumption that $(N, L)$ is regular, we can see that $\widetilde{\psi}$ is continuous.

\end{proof}

Taking the desuspension of (\ref{eq rel}) we get a morphism
\begin{equation} \label{eq mor to J}
    \Sigma^{-V_{\lambda}^0(A_0, g, \bP)} (U_1')^+ 
        \rightarrow 
    U_1^+ \wedge  J(\widetilde{W};A_0, g, \bP).
\end{equation}
Next we define a morphism $J(\widetilde{W}, A_0, g, \bP) \rightarrow \SWF(Y, \fc, g, \bP)$ as follows. 
By Lemma \ref{lem homotopy eq},  we have isomorphisms
\[
     J (\widetilde{W})
       \cong
     \Sigma^{-1} C \big( k:J(W^1_{-n} \cup \cdots \cup W^1_n)  )  \rightarrow
          \Sigma J(W^2_{-n} \cup \cdots \cup W^2_{n} ) \big)
\]
and
\[
  \begin{split}
   &  J(W^1_{-n} \cup \cdots \cup W^1_{n}) \cong J(W^1_{-n}) \vee \cdots \vee J(W^1_{n}),  \\
   &  J(W^2_{-n} \cup \cdots \cup W^2_{n}) \cong
         J(W^2_{-n}) \vee \cdots \vee J(W^2_{n}).
   \end{split}
\]
The following diagram is commutative up to canonical homotopy:
\[
     \xymatrix{
      J(W^1_{-n}) \vee \cdots \vee J(W^1_n) \ar[d]_k \ar[r] & J(W^1_n) \ar[d]^{k^1}  \\
      J(W^2_{-n}) \vee \cdots \vee J(W^2_n) \ar[r]  & J(W^2_n)
     }
\]
Here the morphism $J(W^i_{-n}) \vee \cdots \vee J(W^i_n) \rightarrow J(W^i_n)$ is the morphism induced by $\ff_1$. (More precisely, we need to choose trivialization of a vector space and a vector bundle  as in Section \ref{subsec fi fj} to get the homotopy.) Therefore we get
\begin{equation} \label{eq mor J to SWF}
     J(\widetilde{W}) 
        \rightarrow
     \SWF(Y, \fc, g, \bP).
\end{equation}
Composing this morphism with (\ref{eq mor to J}), we get
\[
      \psi_{X_1, \hat{\fc}_1, g, \bP}:
      \Sigma^{V_{\lambda}^0(A_0, g, \bP)}(U_1')^+ 
          \rightarrow 
      U_1^+ \wedge \SWF(Y,\fc, g, \bP).
\]
Although we assumed that $b_1(Y) = 1$ and $N=2$, the construction can be generalized to any case.
More generally, for each submodule $H \subset H^1(Y;\Z)$ of rank $b_1(Y)$ we can define a morphism
\[
      \psi_{X_1, \hat{\fc}_1, H, g, \bP}:
      \Sigma^{-V_{\lambda}^0(A_0, g, \bP)} (U_1')^+ 
          \rightarrow 
      U_1^+ \wedge \SWF(Y,\fc, H)
\]
in $\fC$.

\begin{prop}
The morphism $\psi_{X_1, \hat{\fc}_1, H, g, \bP}$ is independent of the choices of connection $\hat{A}_1$ with $\hat{A}_1|_Y$ flat, $U_1$, $\lambda, \mu$, Riemannian metric $\hat{g}_1$ with $\hat{g}_1|_{Y} = g$.
\end{prop}

The proof of this proposition is omitted.

\section{Proof of gluing formula}

\subsection{Proof of Theorem \ref{thm gluing}}

To simplify notation, we give the proof in the case $b_1(Y) = 1, N = 2$.
After choosing some data, we may think of $\Sigma^2 \Sigma^{V_{\lambda}^{\mu}} \eta \circ ( \psi_{X_1, \hat{\fc}_1, H, \hat{g}_1, \bP} \wedge \psi_{X_2, \hat{\fc}_2, \hat{g}_2, \bP} )$ as a continuous map
\[
    \Sigma^2 (U_1')^+ \wedge (U_2')^+ \rightarrow
    \Sigma^2 U_1^+ \wedge U_2^+ \wedge (V_{\lambda}^{\mu})^+.
\]
Here we think of $(U_j)^+$ and $(U_j')^+$ as $B(U_j, \epsilon) / S(U_j, \epsilon)$ and $B(U_j', R_j') / S(U_j', R_j')$ for some small $\epsilon > 0$ and large $R'_j > 0$.

\begin{prop} \label{prop y_1 - y_2}
Let $H$ be a submodule of $H^1(Y;\Z)$ generated by $m_1 h_1$, where $h_1$ is a generator of $H^1(Y;\Z)$ and $m_1$ is a positive integer. If $m_1$ is sufficiently large, then the map $\Sigma^2  \Sigma^{V_{\lambda}^{\mu}} \eta \circ (\psi_{X_1, \hat{\fc_1}, H, g, \bP} \wedge \psi_{X_2, \hat{\fc_2}, H, g, \bP})$ is $U(1)$-equivariantly homotopic to the suspension by $\R^2$ of the following map
\begin{equation} \label{eq map y_1 - y_2}
      \begin{array}{rcl}
       (U_1')^+ \wedge ( U_2')^+ & \longrightarrow & 
           U_1^+ \wedge U_2^+ \wedge ( V_{\lambda}^{\mu} )^+ \\
       (\hat{x}_1, \hat{x}_2) & \longmapsto & 
        \left\{
          \begin{array}{ll}
           ( {\displaystyle \prod_{j=1}^2 \pr_{U_j} sw( \hat{x}_j )   },
              y_1 - y_2)&
            \text{if}
              \left\{
                \begin{array}{l}
                  \| \pr_{U_j} sw( \hat{x}_j) \| < \epsilon, \\
                  \| y_1 - y_2 \| < \epsilon,
                \end{array}
              \right.   \\
          *  & \text{otherwise.}
          \end{array}
         \right.
      \end{array}
\end{equation}
Here $y_j = p_{\lambda}^{\mu}  i^* \hat{x}_j$ and $i:Y \hookrightarrow X = X_1 \cup_Y X_2$ is the inclusion, and we consider $U_j^+$ and $(U_j')^+$ to be $B(U_j, \epsilon)/S(U_j, \epsilon)$ and $B(U_j', R_j')/S(U_j,R_j')$.
\end{prop}

\begin{proof}
If $m_1$ is sufficiently large, $p_{\lambda}^{\mu} ( i^* ( B(U_j', R_j') ) ) \subset \widetilde{W}:= W^1_n \cup W_n^2$ for $j = 1, 2$.
We can take regular index pairs $(N, L)$ and $(N, \overline{L})$ for $\Inv ( \widetilde{W} \cap V_{\lambda}^{\mu} , \gamma_{\lambda}^{\mu} )$ and $\Inv ( \widetilde{W} \cap V_{\lambda}^{\mu} , \bar{\gamma}_{\lambda}^{\mu} )$ such that
\[
   \begin{split}
  & i^*( B(U_1', R_1') ) \subset N \backslash L 
         \subset  \widetilde{W} \cap V_{\lambda}^{\mu}, \\
  & i^*( B(U_2', R_2') ) \subset N  \backslash \overline{L} 
        \subset \widetilde{W} \cap V_{\lambda}^{\mu},   \\
  & \text{$N$ is a manifold with boundary $\partial N = L \cup \overline{L}$}, \\
  & \partial L = \partial \overline{L} = L \cap \overline{L}.
   \end{split}
\]
See Lemma \ref{lem N L} and \cite[Section 3.2]{Cornea Homotopical}. ($\bar{\gamma}_{\lambda}^{\mu}$ is the inverse flow of $\gamma_{\lambda}^{\mu}$.)
 
   Take $(t_j, \hat{x}_j) \in \Sigma (U_j')^+$ and put $y_j = p_{\lambda}^{\mu} i^*\hat{x}_j$. 
As in the proof of Lemma \ref{lem CD duality}, we can write
\begin{gather} 
   \Sigma^2 \eta \circ (\psi_{X_1} \wedge \psi_{X_2}) 
              (t_1, t_2, \hat{x}_1, \hat{x}_2)  \nonumber \\
                  \label{eq eta psi}
  = \left\{
       \begin{array}{ll}
         (s_1(\zeta), s_2(\zeta),  \pr_{U_1} sw(\hat{x}_1), \pr_{U_2} sw(\hat{x}_2),
            l_1(\zeta) - l_2(\zeta) ) &
         \text{if} \left\{
              \begin{array}{l}
                \| \pr_{U_j} sw(\hat{x}_j) \|  < \epsilon, \\
                y_1  \cdot [0, T] \subset 
                  N \backslash L,\\
               y_2 \cdot [-T, 0] \subset N \backslash \overline{L}, \\
                \| l_1(\zeta) - l_2(\zeta)) \| < \epsilon,  \\
                \text{$1-s( y_1) \leq t_1 \leq 1$ or}  \\
                 1- \bar{s}(y_2) \leq t_2 \leq 1,   \\
              \end{array}
                  \right.   \vspace{2mm} \\
     * & \text{otherwise.}
       \end{array}
    \right.
\end{gather}
Here $\zeta =(t_1, t_2, \hat{x}_1,  \hat{x}_2)$, $s, \bar{s}:N \rightarrow [0,1]$ and $l_j:N \rightarrow N$ with
\[
      \| l_j(\zeta) - y_j \cdot \tau_j(\zeta) \| \leq O(\delta)
\]
for some continuous function $\tau_j$ of $\zeta$ with $\tau_1 \geq 0, \tau_2 \leq 0$ and with $| \tau_j |$ bounded by a constant independent of $U_j$, $\lambda, \mu$. (The boundedness of $\tau_j$ comes from Lemma \ref{lem f T_0} and Lemma \ref{lem rel inv T_0}.)

We can write
\[
      l_1(\zeta) - l_2(\zeta) = 
      y_1  \cdot \tau_1(\zeta) -
       y_2 \cdot \tau_2(\zeta)
      +O(\delta).
\]
For $u \in [0,1]$, let $H_u$ be a continuous map
\[
   \Sigma^2 (U_1')^+ \wedge (U_2')^+ \rightarrow 
     \Sigma^2 (U_1)^+ \wedge (U_2)^+ \wedge (V_{\lambda}^{\mu})^+
\]
defined by 
\begin{gather*}
     H_u(\zeta) =  \\
      \left\{
       \begin{array}{ll}
         (s_1(\zeta), s_2(\zeta),  \pr_{U_1} sw(\hat{x}_1), \pr_{U_2} sw(\hat{x}_2),
            L_u(\zeta) ) &
         \text{if} \left\{
              \begin{array}{l}
                \| \pr_{U_j} sw(\hat{x}_j) \|  < \epsilon, \\
                 y_1  \cdot [0, (1-u) T] \subset 
                  N \backslash L,\\
               y_2 \cdot [-(1-u)T, 0] \subset N \backslash \overline{L}, \\
                \| L_u(\zeta) \| < \epsilon,  \\
                \text{$1-s( y_1) \leq t_1 \leq 1$ or}  \\
                 1- \bar{s}(y_2) \leq t_2 \leq 1,   \\
              \end{array}
                  \right.   \vspace{2mm} \\
     * & \text{otherwise.}
       \end{array}
    \right.
\end{gather*}
Here
\[  
     L_u(\zeta) = 
      y_1  \cdot (1-u) \tau_1(\zeta) -
      y_2 \cdot (1-u) \tau_2(\zeta)
     + (1-u) O(\delta).
\]
We can prove that $H_u$ is well defined for large $U_1, U_2, -\lambda, \mu$ and small $\delta, \epsilon$ as in \cite[p.130]{Manolescu Gluing} using the compactness of the moduli space of monopoles on a closed 4-manifold.
We can see that
\begin{gather*}
     H_1(\zeta) =   \\
      \left\{
        \begin{array}{ll}
        (  s_1(\zeta), s_2(\zeta), {\displaystyle \prod_{j=1}^2 \pr_{U_j} sw(\hat{x}_j)}, 
         y_1 - y_2)  &
         \text{if}
           \left\{
             \begin{array}{l}
                \| \pr_{U_j} sw(\hat{x}_j) \| < \epsilon, \\
                \| y_1 - y_2  \| < \epsilon, \\ 
                1-s(y_1) \leq t_1 \leq 1  \ \text{or} \\
                1-\bar{s}(y_2) \leq t_2 \leq 1,                
             \end{array}
           \right.   \\
        * & \text{otherwise.}
        \end{array}   
      \right.
\end{gather*}

The same deformation of $s_j(\zeta)$ as that in the proof of Lemma \ref{lem CD duality} gives a homotopy from $H_1$  to the suspension by $\R^2$ of the map (\ref{eq map y_1 - y_2}). 
We have done the proof of Proposition \ref{prop y_1 - y_2}.
\end{proof}

\noindent
{\it Proof of Theorem \ref{thm gluing}} 

Although we used a different boundary condition to define the relative invariants from that of \cite{Manolescu Gluing}, we can apply the proof of the gluing formula in \cite[Section 4]{Manolescu Gluing} to (\ref{eq map y_1 - y_2})  with some modification (\cite{Mano errata}) and we have done the proof of Theorem \ref{thm gluing}. 
\qed


\end{document}